\documentclass[11pt]{amsart}

\usepackage[english]{babel}
\usepackage{hyperref}
\usepackage{amsmath}
\usepackage{amsfonts}
\usepackage{amssymb}
\usepackage{amsthm}
\usepackage[initials]{amsrefs}
\usepackage{thmtools}
\usepackage{color}
\usepackage{tikz}
\usepackage{tikz-cd}
\usepackage{tikz-3dplot}
\usepackage{enumitem}
\usetikzlibrary{babel, intersections, decorations.markings, decorations.pathreplacing}
\usepackage{changepage} 

\allowdisplaybreaks

\setlist[enumerate,1]{label={\arabic*.}}

\declaretheorem[name=Theorem, style=plain, numberwithin=section]{satz}
\declaretheorem[name=Lemma, style=plain, sibling=satz]{lem}
\declaretheorem[name=Corollary, style=plain, sibling=satz]{kor}

\declaretheorem[name=Theorem, style=plain, numbered=no]{satzA}

\newcommand{\N}{\mathbb{N}}
\newcommand{\Z}{\mathbb{Z}}

\newcommand{\R}{\mathbb{R}}

\renewcommand{\S}{\mathbb{S}}

\newcommand{\Vol}{\operatorname{Vol}}

\renewcommand{\epsilon}{\varepsilon}

\newcommand{\Isom}{\operatorname{Isom}}
\newcommand{\id}{\operatorname{id}}

\newcommand{\injrad}{\operatorname{Inj \, Rad}}
\newcommand{\tors}{\operatorname{tors}}

\newcommand{\K}{\mathbb{K}}
\newcommand{\T}{\mathbb{T}}
\renewcommand{\H}{\mathbb{H}}

\newcommand{\FIX}{\operatorname{Fix}}

\newcommand{\im}{\operatorname{im}}
\newcommand{\grad}{\operatorname{grad}}

\begin{document}

\title{Topological complexity of visibility manifolds}

\author{Hartwig Senska}
\address{Karlsruhe Institute of Technology, Germany}
\email{hartwig.senska@kit.edu}

\begin{abstract}
A recent result of Bader, Gelander and Sauer shows that for manifolds of pinched negative curvature, the torsion part of the homology can be controlled by the volume. This is done by constructing an efficient simplicial model of the thick part, which also provides another proof of the analogous statement for the free part of the homology, a classical theorem due to Gromov.

We will extend these results to more general curvature conditions, namely the case where the sectional curvature can get arbitrarily close to zero, but the visibility axiom still holds.
\end{abstract}
\maketitle

\section{Introduction}

A classical theorem of Gromov \cite{BGS} says that for a Hadamard $n$-manifold $X$ with pinched negative sectional curvature $-1 \leq K \leq a < 0$ (for some $a < 0$) and a lattice $\Gamma < \Isom(X)$, the free part of the homology -- encoded by the Betti numbers -- grows at most linearly, i.e.
\[
b_k(X\slash\Gamma) \leq C \cdot \Vol(X\slash\Gamma)
\]
for all $k=0,\ldots,n$, where $C = C(n) > 0$ is a constant depending only on the dimension $n$ (the Betti numbers can be taken with respect to an arbitrary field). This statement can be extended to non-positive curvature $-1 \leq K \leq 0$ under some additional assumptions\footnote{Namely in the analytic case and the absence of Euclidean de Rham factors of $X$.}. As the homology splits into the free part and the torsion part, it remained an open question whether a similar statement holds true for the torsion in homology. This was answered positively by Bader, Gelander and Sauer in \cite{Sauer}: again in the pinched negative curvature situation, we have
\[
\log |\tors H_k(X\slash\Gamma; \Z)| \leq C \cdot \Vol(X\slash\Gamma)
\]
for all $k=0,\ldots,n$ and some constant $C = C(n) > 0$ depending only on $n$, although the case of degree $k=1$ in dimension $n=3$ has to be excluded; a specific counterexample for the latter case can be constructed using Dehn surgery and is explicitly given in \cite{Sauer}. While Gromov's proof for the Betti numbers is Morse theoretic, the statement for the torsion homology is a consequence of an efficient simplicial decomposition of the thick part of $X\slash\Gamma$.

In this paper, we will construct a similar efficient simplicial model of the thick part under more general curvature assumptions. Let $X$ be Hadamard with curvature $-1 \leq K < 0$ and visibility axiom (recall that pinched negative curvature $-1 \leq K \leq a < 0$ is a special case of this). For an arbitrary lattice $\Gamma < \Isom(X)$ denote the corresponding quotient manifold by $M := X\slash\Gamma$ and write $M_+$ for its thick part. Our main result says:

\begin{satzA}[see Theorem \ref{Satz:Hauptresultat_Sichtbarkeit}]
There exist constants $C = C(n) > 0$ and $D = D(n) > 0$ depending only on the dimension $n$, such that for any such manifold $M$, the pair $(M_+, \partial M_+)$ -- i.e. the thick part and its boundary -- is as a pair homotopy equivalent to a simplicial pair $(S, S')$, where the number of vertices of $S$ is bounded by $C \cdot \Vol(M)$ and the degree at the vertices of $S$ is universally bounded by $D$.
\end{satzA}

Using the Mayer-Vietoris sequence and general properties of the thick-thin decomposition, we immediately get a linear bound on the free part in homology.

\begin{satzA}[see Theorem \ref{Satz:Freier_Anteil_beschraenkt_durch_Volumen}]
There is a constant $E = E(n) > 0$ depending only on the dimension $n$, such that for any such manifold $M$ we have
\[
b_k(M;\K) \leq E \cdot \Vol(M)
\]
for all $k \in \N_0$ and arbitrary coefficient field $\K$.
\end{satzA}

Note that the present situation is not completely covered by either of the Gromov statements: our curvature assumptions are more general than the pinched case, and the non-positive curvature case needed analyticity. If $\operatorname{char}(\K) = 0$, the above statement is included in the results of Samet \cite{Samet}, so the interesting new information is provided for $\operatorname{char}(\K) \neq 0$.

By the same reasoning as in \cite{Sauer}, we can also deduce the corresponding statement for the torsion part of the homology.

\begin{satzA}[see Theorem \ref{Satz:Torsionsanteil_beschraenkt_durch_Volumen}]
There is a constant $F = F(n) > 0$ depending only on the dimension $n$, such that for any such manifold $M$ we have
\[
\log | \tors H_k(M;\Z) | \leq F \cdot \Vol(M)
\]
for all $k \in \N_0$, where for $n=3$ the case $k=1$ has to be excluded.
\end{satzA}

A last application of the efficient simplicial model for the thick part is the following counting statement, which again is analogous to a result in \cite{Sauer}.

\begin{satzA}[see Theorem \ref{Satz:Homotopie-und_Homoeomorphietypen_zaehlen}]
Let $\mathfrak{Htp}_n(V)$ denote the number of homotopy classes of complete visibility $n$-manifolds of volume at most $V$, where $n \geq 4$. Then $\log \mathfrak{Htp}_n(V)$ grows in the magnitude of $V \cdot \log V$.
\end{satzA}

A crucial tool throughout the paper will be the thick-thin decomposition of negatively curved manifolds, which we state and prove in a slightly generalized form (see Theorem \ref{Satz:DickDuennZerlegungSichtbarkeit}) in section \ref{Kapitel:Dick-duenn-Zerlegung}.

Although the classes of negatively curved visibility manifolds and of negatively pinched manifolds are not expected to be equal (up to homeomorphism or diffeomorphism), no explicit example separating these classes is known. In section \ref{Kapitel:Beispiele_Sichtbarkeit}, we will at least give examples of finite volume visibility manifolds that are not pinched, generalizing a construction of \cite{JiWu}; note that since they are constructed out of hyperbolic manifolds by perturbing the metric on the cusps, this does not provide an answer to the question at hand.

In an upcoming paper, we prove similar statements to the above for hyperbolic orbifolds. All these results -- both for visibility manifolds and for hyperbolic orbifolds -- are already available in German in \cite{SenskaDissertation}.

\subsection{Structure of the paper}

The next section \ref{Kapitel:Grundlagen} reviews some basic concepts. After that, in section \ref{Kapitel:Dick-duenn-Zerlegung}, we state and prove a slightly generalized thick-thin decomposition. Section \ref{Kapitel:Hauptresultat} contains the main result and its applications. The final section \ref{Kapitel:Beispiele_Sichtbarkeit} is dedicated to constructing examples of negatively curved visibility manifolds which do not have pinched negative curvature.

\subsection{Acknowledgement}

The results presented here are part of my doctoral thesis \cite{SenskaDissertation}. I want to express gratitude to my advisor Prof. Roman Sauer and the DFG for supporting my work via the RTG 2229.

\section{Preliminaries}\label{Kapitel:Grundlagen}

We will first fix some notation and review basic properties of negative curvature. Main references for most of the statements given in this section are \cite{BGS} and \cite{BridsonHaefliger} (see also \cite{SenskaDissertation} Kapitel 1).

An $n$-dimensional \textbf{Hadamard manifold} $X$ is a complete, simply connected Riemannian manifold of non-positive sectional curvature $K \leq 0$. By the Hadamard-Cartan theorem, such an $X$ is always diffeomorphic to $\R^n$. For any complete Riemannian manifold $M$ with $K \leq 0$, its universal cover $\widetilde{M}$ is a Hadamard manifold. Assuming further that $M$ has finite volume, we get that $M = X\slash\Gamma$, where $X = \widetilde{M}$ and $\Gamma < \Isom(X)$ is a lattice; conversely, every lattice $\Gamma < \Isom(X)$ yields such a manifold $M$.

If $X$ is Hadamard, then its distance function $d: X \times X \rightarrow [0,\infty)$ is convex; this turns out to be a remarkably powerful property. For a closed convex set $W \subseteq X$, there is a well-defined \textbf{projection} $\pi_W: X \rightarrow W$ sending a point $x \in X$ to the (unique) point $\pi_W(x) \in W$ of smallest distance to $x$; we will call $\pi_W(x)$ the \textbf{projection point} or \textbf{foot point} of $x$ in $W$. This projection is equivariant under isometries preserving $W$, i.e. if $\gamma \in \Isom(X)$ with $\gamma W = W$, then $\pi_W(\gamma x) = \gamma \pi_W(x)$ for all $x \in X$.

Many phenomena in negative curvature are related to the boundary at infinity, which in some sense describes the behavior outside of arbitrarily large compact sets. Two geodesics $c, c': \R \rightarrow X$ are \textbf{asymptotic} if there is a constant $C > 0$ such that $d(c(t), c'(t)) < C$ for all $t \geq 0$. This defines an equivalence relation on the set of geodesics in $X$, where $c(\infty)$ denotes the equivalence class of $c$ (similarly, $c(-\infty)$ denotes the equivalence class of $c$ when parametrized in the other direction); the set of equivalence classes is the \textbf{boundary at infinity} $X(\infty)$. We can endow $\overline{X} := X \cup X(\infty)$ with a topology that coincides with the usual one on $X$ and such that $X(\infty) \cong S^{n-1}$. With this topology, $\overline{X}$ is topologically a closed ball. We say that $X$ satisfies the \textbf{visibility axiom} if for every choice of different boundary points $z, z' \in X(\infty)$, there is a connecting geodesic $c: \R \rightarrow X$ with $z = c(\infty)$, $z' = c(-\infty)$. A manifold whose universal cover satisfies the visibility axiom is a \textbf{visibility manifold}. If the sectional curvature is bounded away from 0 (i.e. $K \leq a < 0$ for some $a < 0$), the visibility axiom is satisfied. From now on, we will only study visibility manifolds, where additionally $K<0$.

We can try to extend the concept of balls around points $x \in X$ to points $z \in X(\infty)$ at infinity; this leads to the notion of horoballs. Let $c$ be a geodesic with $c(\infty) = z$ and define the \textbf{Busemann function} $h_c: X \rightarrow \R$ via $h_c(x) := \lim_{t\rightarrow\infty} (d(x,c(t)) - t)$, which is well-defined, convex and $C^2$. An open \textbf{horoball} around $z$ is a sublevel set $HB = \{ h_c < a \}$ (for some $a \in \R$); similarly, $\{ h_c \leq a \}$ is a closed horoball. Their boundary $HS := \partial HB = h_c^{-1}(\{a\})$ is called \textbf{horosphere}. Each geodesic $c'$ with endpoint $c'(\infty) = z$ intersects each horosphere around $z$ orthogonally in a unique point.

Every isometry $\gamma \in \Isom(X)$ gives rise to a \textbf{displacement function}
\[
d_{\gamma}: X \rightarrow [0,\infty), \qquad x \mapsto d_{\gamma}(x) := d(x, \gamma x).
\]
We can classify the nontrivial isometries of $\Isom(X)$ by the behavior of their displacement functions: $\gamma$ is \textbf{elliptic} if $d_{\gamma}$ has minimum $0$; it is \textbf{hyperbolic} if $d_{\gamma}$ has minimum $>0$; and it is \textbf{parabolic} if $d_{\gamma}$ has no minimum. If $M = X\slash\Gamma$ is a manifold, then $\Gamma < \Isom(X)$ contains no elliptic isometries (and is thus torsion-free); if $M$ is non-compact, there has to be a parabolic $\gamma \in \Gamma$. The different isometry types are stable under taking powers (with powers $\neq 0$). From now on, we will always exclude elliptic isometries, as we are only interested in manifolds. In case $X$ satisfies the visibility axiom, we can alternatively classify the isometries via their fixed points in $X(\infty)$: hyperbolic isometries have precisely two fixed points in $X(\infty)$, whereas parabolic isometries have precisely one. For a hyperbolic $\gamma$, the minimal set $\{ d_{\gamma} = \min d_{\gamma} \}$ consists of geodesics between the two boundary fixed points, and $\gamma$ acts via translation on them; if $K<0$ there is only one such geodesic, namely the \textbf{axis} of $\gamma$. In a discrete subgroup $G < \Isom(X)$, parabolic and hyperbolic isometries can not have common fixed points; moreover, for two hyperbolic isometries in such a discrete $G$, if they have one common fixed point, their second fixed point also has to coincide.

Sublevel sets $\{ d_{\gamma} < a \}$ (for some $a \in \R$) of the displacement function will play a prominent role in this paper. Note that these are convex sets and we have $\gamma' \{ d_{\gamma}  < a \} = \{ d_{\gamma' \gamma \gamma'^{-1}} < a \}$; moreover, $\inf_{x \in X} d_{\gamma}(x) = \inf_{x \in X} d_{\gamma' \gamma \gamma'^{-1}}(x)$. For a set $A \subseteq X$ and $r \geq 0$, we will use $(A)_r = \{ x \in X : d(x,A) < r \}$ to denote the open $r$-neighborhood of $A$.

We will finish with some monotonicity statements. If $\gamma \in \Isom(X)$ is parabolic with parabolic fixed point $z \in X(\infty)$ and $c$ a geodesic with $c(\infty) = z$, then the map $t \mapsto d_{\gamma}(c(t))$ is strictly decreasing; note that in our case of $K < 0$ with visibility axiom, the limit for $t \rightarrow \infty$ is not necessarily $0$ (whereas in the $K \leq a < 0$ case it indeed is). Similarly, for a hyperbolic isometry $\gamma$ with axis $A$ and a geodesic ray $c$ from $\pi_A(x) \in A$ to $x \notin A$, the function $t \mapsto d_{\gamma}(c(t))$ is strictly increasing with limit $\infty$ for $t \rightarrow \infty$. Using this monotonicity, we immediately see that a (nonempty) sublevel set $\{ d_{\gamma} < a \}$ is a tubular neighborhood of the axis $A$, and is contained in the $r$-neighborhood $(A)_r$ of $A$ for some $r > 0$.

\section{Thick-thin decomposition}\label{Kapitel:Dick-duenn-Zerlegung}

In this section we will give a proof of a slightly generalized thick-thin decomposition, where we might weight different isometries differently. This construction was already used in the pinched curvature case in \cite{Sauer}, but without proof. The result will be no news to the expert and can easily be skipped, as the statements are basically identical to the ones e.g. in \cite{BGS} chapter 10, \cite{BenPet} chapter D or \cite{Bowditch} chapter 3.5, and the proofs are almost so: we just have to account for the more general curvature assumptions and the variable weighting of the isometries. Thus, we see this section more as a convenient and concise summary and not as novel result in itself.

Let $X$ be an $n$-dimensional Hadamard manifold with sectional curvature $-1 \leq K < 0$ and visibility axiom. Moreover, $\Gamma $ is a torsion-free lattice in $\Isom(X)$ and $M := X\slash\Gamma$, i.e. $M$ is a complete, finite volume visibility manifold with $-1 \leq K < 0$ and $\pi_1(M) \cong \Gamma$.

Let $\epsilon(n)$ be the Margulis $\epsilon$ (see Theorem \ref{Samet2.1} below), $\epsilon \in (0, \epsilon(n)/2]$ fixed and choose a conjugation invariant assignment $\Gamma \setminus \{\id\} \rightarrow [\epsilon, \epsilon(n)/2]$, $\gamma \mapsto \epsilon_{\gamma}$. The \textbf{thin part} $X_-$ of $X$ is then defined as
\[
X_- := \bigcup_{\gamma \in \Gamma \setminus\{\id\}} \{ d_{\gamma} < \epsilon_{\gamma} \}
\]
and the \textbf{thick part} as its complement $X_+ := X \setminus X_-$. Note that $X_-$ also depends on $\Gamma$ and the assignment $\gamma \mapsto \epsilon_{\gamma}$, although this is omitted in the notation. Similarly, we call $M_- := X_-\slash\Gamma$ the \textbf{thin part} of $M$ and its complement $M_+ := X_+\slash\Gamma$ the \textbf{thick part} (this is well-defined by Lemma \ref{Lem:TransformationNiveaumengen} and \ref{Lem:BGS10.2Analogon} below). The function
\[
d_{\Gamma}: X \rightarrow \R, \qquad x \mapsto d_{\Gamma}(x) := \inf\limits_{\gamma \in \Gamma \setminus \{\id\}} d_{\gamma}(x)
\]
gives us the useful relation $d_{\Gamma}(x) = 2 \cdot \injrad_M(\pi(x))$ between the displacement in $X$ and the injectivity radius in $M$, where $\pi: X \rightarrow X\slash\Gamma = M$ denotes the projection. Finally, recall that for a subset $S \subseteq \overline{X}$, we write $\Gamma_S = \{ \gamma \in \Gamma : \gamma S = S \}$ for the setwise stabilizer group.

The crucial tool in the thick-thin decomposition is the Margulis lemma.

\begin{satz}[Margulis lemma; \cite{Samet} Theorem 2.1]\label{Samet2.1}
There are constants $\epsilon(n) > 0$ and $m(n) \in \N$ depending only on $n$, such that if $X$ is an $n$-dimensional Hadamard manifold with sectional curvature $-1 \leq K \leq 0$, then for every discrete group $\Gamma < \Isom(X)$, every $x \in X$ and every $\epsilon \leq \epsilon(n)$, the group
\[
\Gamma_{\epsilon}(x) := \langle \{ \gamma \in \Gamma : d_{\gamma}(x) < \epsilon \} \rangle
\]
contains a nilpotent normal subgroup $N$ of index $\leq m(n)$. If $\Gamma_{\epsilon}(x)$ is finite, then $N$ is abelian.
\end{satz}

The constants $\epsilon(n)$ and $m(n)$ in Theorem \ref{Samet2.1} will be called Margulis $\epsilon$ and Margulis index constant, respectively.

Sublevel sets of the displacement function behave as follows, which shows why we need the conjugation invariance of the assignment $\gamma \mapsto \epsilon_{\gamma}$.

\begin{lem}\label{Lem:TransformationNiveaumengen}
Let $\gamma, \gamma' \in \Isom(X)$ and $a \geq 0$. Then $\gamma \{ d_{\gamma'} < a \} = \{ d_{\gamma \gamma' \gamma^{-1}} < a \}$ and similarly for $\{ d_{\gamma'} \leq a \}$.
\end{lem}

The following lemma serves as an analog for \cite{BGS} Lemma 10.2.

\begin{lem}\label{Lem:BGS10.2Analogon}
Let $W \subseteq X_-$ be a connected component of $X_-$. Then:
\begin{enumerate}
\item $W$ is precisely invariant under $\Gamma$, i.e. for $\gamma \in \Gamma$ either $\gamma W = W$ or $\gamma W \cap W = \emptyset$.

\item If $\gamma \in \Gamma$ and $x\in W$ such that $d_{\gamma}(x) < \epsilon_{\gamma}$, then $\gamma \in \Gamma_W = \{ \gamma \in \Gamma : \gamma W = W \}$.
\end{enumerate}
\end{lem}

\begin{proof}
\begin{enumerate}
\item Let $\gamma_0 \in \Gamma$. By Lemma \ref{Lem:TransformationNiveaumengen} and the conjugation invariance of $\gamma \mapsto \epsilon_{\gamma}$, we get
\begin{align*}
\gamma_0 \bigcup\limits_{\gamma \in \Gamma \setminus \{\id\}} \{ d_{\gamma} < \epsilon_{\gamma} \} &= \bigcup\limits_{\gamma \in \Gamma \setminus \{\id\}} \{ d_{\gamma_0 \gamma \gamma_0^{-1}} < \epsilon_{\gamma} \} \\
&= \bigcup\limits_{\gamma \in \Gamma \setminus \{\id\}} \{ d_{\gamma_0 \gamma \gamma_0^{-1}} < \epsilon_{\gamma_0 \gamma \gamma_0^{-1}} \} \\
&= \bigcup\limits_{\gamma' \in \Gamma \setminus \{\id\}} \{ d_{\gamma'} < \epsilon_{\gamma'} \},
\end{align*}
so $\gamma_0 X_- = X_-$. Thus connected components will be mapped to connected components and $W$ is precisely invariant under $\Gamma$.

\item By assumption, $d_{\gamma} (\gamma x) = d_{\gamma}(x) < \epsilon_{\gamma}$, so $\gamma x \in \{ d_{\gamma} < \epsilon_{\gamma} \}$. Since this set is convex, the geodesic from $x$ to $\gamma x$ is also contained in it; thus $x$ and $\gamma x$ lie in the same connected component, i.e. $\gamma x \in W$. So $\gamma W \cap W \supseteq \{ \gamma x \} \neq \emptyset$ and by 1. we conclude $\gamma W = W$.
\end{enumerate}
\end{proof}

In particular, $\pi(W) \cong W\slash\Gamma_W$ for any connected component $W$ of $X_-$.

Virtually nilpotent groups only contain one type of isometries:

\begin{lem}[\cite{Eberlein} Lemma 3.1b]\label{Eberlein3.1b}
Let $N < \Gamma$ be virtually nilpotent. Then $N \subseteq \Gamma_z = \{ \gamma \in \Gamma : \gamma z = z \}$ for some $z \in X(\infty)$. In particular, nontrivial elements of $N$ are either all hyperbolic or all parabolic. Furthermore, $\FIX(\gamma) = \FIX(\gamma')$ for all nontrivial $\gamma, \gamma' \in N$, where $\FIX(\gamma) := \{ z \in X(\infty) : \gamma z = z \}$.
\end{lem}

Because of the following lemma, we will always choose $\gamma \mapsto \epsilon_{\gamma}$ in such a way that $\epsilon_{\gamma} \in [\epsilon, \epsilon(n)/2]$.

\begin{lem}[\cite{Eberlein} Lemma 3.1c]\label{Eberlein3.1c}
Let $A \subseteq X$ be path connected and such that $d_{\Gamma}(x) \leq \epsilon(n)/2$ for all $x \in A$. If $\gamma, \gamma' \in \Gamma$ are nontrivial with $d_{\gamma}(x) \leq \epsilon(n)/2$ and $d_{\gamma'}(x') \leq \epsilon(n)/2$ for suitable points $x, x' \in A$, then $\FIX(\gamma) = \FIX(\gamma')$.
\end{lem}

\begin{kor}\label{Kor:Eberlein3.1cAnalogon}
Let $W$ be a connected component of $X_-$ and assume the assignment $\gamma \mapsto \epsilon_{\gamma}$ is chosen such that $\epsilon_{\gamma} \in [\epsilon, \epsilon(n)/2]$ for all $\gamma \in \Gamma \setminus \{\id\}$. If $\gamma, \gamma' \in \Gamma$ are nontrivial with $d_{\gamma}(x) < \epsilon_{\gamma}$ and $d_{\gamma'}(x') < \epsilon_{\gamma'}$ for suitable points $x, x' \in W$, then $\FIX(\gamma) = \FIX(\gamma')$.
\end{kor}

From now on, $W$ will always denote a connected component of the thin part $X_-$. Similar to virtually nilpotent groups, stabilizer groups of such a $W$ only contain one type of isometries.

\begin{lem}\label{Lem:Gamma_W_hyperbolisch_parabolisch}
The nontrivial elements of $\Gamma_W$ are either all hyperbolic with common axis $A$ or all parabolic with common fixed point $z \in X(\infty)$.
\end{lem}

\begin{proof}
The proof is similar to the one in \cite{BGS} Lemma 10.3, though we have to replace the statements (*) and Lemma 10.2 by our analogs Corollary \ref{Kor:Eberlein3.1cAnalogon} and Lemma \ref{Lem:BGS10.2Analogon}. Let
\[
\mathcal{B} := \{ \gamma \in \Gamma \setminus \{\id\} : \text{There is } x \in W \text{ with } d_{\gamma}(x) < \epsilon_{\gamma} \}.
\]
By Corollary \ref{Kor:Eberlein3.1cAnalogon}, $\FIX(\gamma) = \FIX(\gamma')$ for any $\gamma, \gamma' \in \mathcal{B}$. Since parabolic isometries have precisely one fixed point and hyperbolic isometries have two, the elements of $\mathcal{B}$ are either all hyperbolic with common axis $A$ or all parabolic with common fixed point $z \in X(\infty)$. Note that by Lemma \ref{Lem:BGS10.2Analogon} we have $\mathcal{B} \subseteq \Gamma_W$.

If $\gamma \in \Gamma_W$ is nontrivial and $x \in W$, then $\gamma x \in W$. Thus by definition of $X_- \supseteq W$, there is $\beta \in \Gamma \setminus \{\id\}$ such that $\gamma x \in \{ d_{\beta} < \epsilon_{\beta} \}$. Obviously, $\beta \in \mathcal{B}$ and we conclude
\[
\epsilon_{\beta} > d_{\beta}(\gamma x) = d_{\gamma^{-1} \beta \gamma}(x).
\]
By conjugation invariance of $\gamma \mapsto \epsilon_{\gamma}$, it follows that $\epsilon_{\gamma^{-1} \beta \gamma} = \epsilon_{\beta}$ and thus $\gamma^{-1} \beta \gamma \in \mathcal{B}$.

If $\beta$ is hyperbolic with axis $A$, then $\gamma^{-1} \beta \gamma \in \mathcal{B}$ is hyperbolic with axis $\gamma^{-1} A$. Since the axes of two hyperbolic elements of $\mathcal{B}$ have to coincide, we get $\gamma^{-1} A = A$. So $\gamma$ leaves $A$ invariant and thus is itself hyperbolic with axis $A$, see \cite{BGS} Lemma 6.5.

Let otherwise $\beta$ be parabolic with fixed point $z \in X(\infty)$, so $\mathcal{B}$ consists only of parabolic isometries with fixed point $z$. Using $\gamma^{-1} \beta \gamma \in \mathcal{B}$ we get $\gamma^{-1} \beta \gamma (z) = z$, thus $\beta (\gamma z) = \gamma z$, i.e. $\gamma z$ is another fixed of $\beta$. Now uniqueness yields $\gamma z = z$. Since for discrete $\Gamma$, parabolic and hyperbolic isometries can't share a fixed point, we conclude that $\gamma$ itself has to be parabolic.
\end{proof}

The following statements will be useful when dealing with parabolic groups.

\begin{lem}\label{Lem:Eberlein3.1ef}
Let $z \in X(\infty)$ be the fixed point of a parabolic isometry in $\Gamma$.
\begin{enumerate}
\item {\normalfont(\cite{Eberlein} Lemma 3.1e)} For every $\delta > 0$, there is a horosphere $HS$ around $z$ which is precisely invariant under $\Gamma$, such that for every point $x$ in the corresponding closed horoball $HB$, we have $d_{\Gamma}(x) < \delta$.

\item {\normalfont(\cite{Eberlein} Lemma 3.1f)} If $c$ is a geodesic in $X$ with $c(\infty) = z$, there is some $t_1 \in \R$ with $d_{\Gamma}(c(t_1)) > \epsilon(n)/2$.
\end{enumerate}
\end{lem}

\begin{lem}\label{Lem:ParabolischeGruppenStabilisator}
In the parabolic case we have $\Gamma_W = \Gamma_z$, where $z \in X(\infty)$ is the common fixed point of the parabolic elements of $\Gamma_W$.
\end{lem}

\begin{proof}
Lemma \ref{Lem:Gamma_W_hyperbolisch_parabolisch} yields $\Gamma_W \subseteq \Gamma_z$. Let otherwise $\gamma \in \Gamma_z$. Using Lemma \ref{Lem:Eberlein3.1ef}, there is a horoball $HB \subseteq W$ around $z$ which is precisely invariant under $\Gamma$ such that $d_{\Gamma}(x) < \epsilon$ for all $x \in HB$. Since $\gamma z = z$, the horoball $HB' = \gamma HB$ is also centered at $z$, thus has to intersect $HB$ nontrivially. So $HB = HB'$ by the precise invariance of $HB$. But using $HB \subseteq W$, this means that $\gamma W \cap W \neq \emptyset$, so $\gamma W = W$ by Lemma \ref{Lem:BGS10.2Analogon}, i.e. $\gamma \in \Gamma_W$.
\end{proof}

Note that the components $U$ of the thin part $M_-$ are precisely of the form $\pi(W)$ $(\cong W\slash\Gamma_W)$ for some component $W$ of $X_-$. The dichotomy given by Lemma \ref{Lem:Gamma_W_hyperbolisch_parabolisch} will then allow us to study the structure of such $U$ more easily.

\begin{lem}\label{Lem:U_unbeschraenkt_GDW_hyperbolisch}
A component $U$ of $M_-$ is bounded if and only if the hyperbolic case occurs in Lemma \ref{Lem:Gamma_W_hyperbolisch_parabolisch}. (So equivalently, $U$ is unbounded iff the parabolic case happens.)
\end{lem}

\begin{proof}
The proof is basically identical to the one in \cite{BGS} Lemma 10.3 (see the last two paragraphs in the proof there).
\end{proof}

As usual, the bounded components of $M_-$ are known as \textbf{tubes}.

\begin{lem}\label{Lem:Roehren}
Let $U = \pi(W)$ be bounded. Then $U$ is homeomorphic to a tubular neighborhood of a closed geodesic of length $< \epsilon(n)/2$. (Equivalently, $U$ is homeomorphic to a $D^{n-1}$-bundle over $\S^1$ and thus $U \simeq \S^1$.)
\end{lem}

\begin{proof}
Note that $W$ is a finite union of sublevel sets $\{ d_{\gamma} < \epsilon_{\gamma} \}$, where each such $\gamma$ is of the form $\gamma = \gamma_0^k$ for some $k \in \Z \setminus \{0\}$, with $\gamma_0$ the hyperbolic isometry with minimal translation on the common axis $A$. Using that these sublevel sets are tubular neighborhoods of $A$ and the monotonicity of each $d_{\gamma}$ when moving towards $A$, we see that $W$ itself is also a tubular neighborhood of $A$. The statement for $U$ follows after projecting to $M$.
\end{proof}

An unbounded component of $M_-$ is called \textbf{cusp}:

\begin{lem}\label{Lem:Spitzen}
Let $U = \pi(W)$ be unbounded. Then $U$ is homeomorphic to $V \times (0,\infty)$ for a suitable compact, $(n-1)$-dimensional manifold $V$. Moreover, there is a strong deformation retraction of $U$ onto $\partial U$, where $\partial U$ denotes the boundary of $U$ in $M = X\slash\Gamma$.
\end{lem}

\begin{proof}
This is again a fairly standard argument (see also \cite{Eberlein} Theorem 3.1). Let $z \in X(\infty)$ be the parabolic fixed point of $\Gamma_W = \Gamma_z$. If $HS$ is a horosphere around $z$, then $V := HS\slash\Gamma_z$ is a compact $(n-1)$-dimensional $C^2$-manifold by \cite{Eberlein} Lemma 3.1g. Now observe that if $B(t)$ denotes image of $\partial W$ under the geodesic flow to $z$ at time $t$, we also get $B(t)\slash\Gamma_z \cong V$. Finally, $W = \bigcup_{t>0} B(t)$ yields the desired product structure.

The strong deformation retraction of $W$ onto $\partial W$ (which gives us the one for $U$ onto $\partial U$ after projecting to $M$) can be constructed using the above geodesic flow -- now in the opposite direction, i.e. flowing away from $z$ -- and stopping at $\partial W$.
\end{proof}

By our choice of the constants, the components of $X_-$ have a uniform minimal distance from each other.

\begin{lem}\label{Lem:AbstandKomponentenW}
There is a constant $\delta = \delta(n) > 0$ depending only on $n$, such that $d(W,W') \geq \delta$ for all connected components $W \neq W'$ of $X_-$.
\end{lem}

\begin{proof}
Set $\delta(n) := \frac{\epsilon(n)}{8}$ and write $\delta = \delta(n)$. Let $x \in W$ and $x' \in W'$, i.e. $x \in \{ d_{\gamma} < \epsilon_{\gamma} \}$ and $x' \in \{ d_{\gamma'} < \epsilon_{\gamma'} \}$ for suitable $\gamma \in \Gamma_W$ and $\gamma' \in \Gamma_{W'}$. If $d(x,x') < \delta$, then using $\epsilon_{\gamma}, \epsilon_{\gamma'} \leq \epsilon(n)/2$ we get
\[
d_{\gamma}(x') \leq 2 d(x,x') + d_{\gamma}(x) < 2 \delta + \epsilon_{\gamma} \leq \frac{2 \epsilon(n)}{8} + \frac{\epsilon(n)}{2} < \epsilon(n).
\]
Thus $\gamma \in \Gamma_{\epsilon(n)}(x')$ (and trivially $\gamma' \in \Gamma_{\epsilon(n)}(x')$, since $d_{\gamma'}(x') < \epsilon_{\gamma'} < \epsilon(n)$). But $\Gamma_{\epsilon(n)}(x')$ is virtually nilpotent by Theorem \ref{Samet2.1}, so Lemma \ref{Eberlein3.1b} yields $\Gamma_{\epsilon(n)}(x') \subseteq \Gamma_z$ for some $z \in X(\infty)$. Hence $\gamma, \gamma' \in \Gamma_z$.

If $\gamma$ is parabolic, then $\Gamma_W = \Gamma_z$ (Lemma \ref{Lem:Gamma_W_hyperbolisch_parabolisch} and \ref{Lem:ParabolischeGruppenStabilisator}), so $\gamma' \in \Gamma_W$. Using
\[
W = \bigcup\limits_{\gamma \in \Gamma_W \setminus \{\id\}} \{ d_{\gamma} < \epsilon_{\gamma} \},
\]
we deduce $\{ d_{\gamma'} < \epsilon_{\gamma'} \} \subseteq W$ and thus $x' \in W$, i.e. $W = W'$.

Conversely, if $\gamma$ is hyperbolic, let $A$ be the common axis of the nontrivial elements of $\Gamma_W$. Since $\gamma \in \Gamma_z$, we can assume $z = A(\infty)$. By discreteness of $\Gamma$, the nontrivial elements of $\Gamma_z$ all have to be hyperbolic and share their fixed points with $\gamma$, i.e. $A(\infty)$ and $A(-\infty)$. Hence every nontrivial $\gamma'' \in \Gamma_z$ -- and thus in particular $\gamma' \in \Gamma_z$ -- also has axis $A$. Since $\{ d_{\gamma'} < \epsilon_{\gamma'} \}$ is connected and contains $A$, it lies in the same connected component of $X_-$ as $A$, i.e. $\{ d_{\gamma'} < \epsilon_{\gamma'} \} \subseteq W$. So $x' \in W$ and thus $W = W'$.
\end{proof}

Using the previous Lemma \ref{Lem:AbstandKomponentenW} we can deduce that the number of components of $M_-$ is linearly controlled by the volume of $M$:

\begin{lem}\label{Lem:LineareSchrankeAnzahlKomponenten}
There is a constant $C = C(\epsilon, n) > 0$ depending only on $\epsilon$ and $n$, such that the number of connected components of $M_-$ is bounded by $C \cdot \Vol(M)$.
\end{lem}

\begin{proof}
The proof is again a standard argument, although the variable choice of levels $\epsilon_{\gamma} \in [\epsilon, \epsilon(n)/2]$ will create an additional dependence on $\epsilon$.

Let $U \neq U'$ be connected components of $M_-$; recall that $U = W\slash\Gamma_W$ and $U' = W' \slash \Gamma_{W'}$ for suitable components $W \neq W'$ of $X_-$. By Lemma \ref{Lem:AbstandKomponentenW} there is $\delta := \delta(n) > 0$ such that $d(W, W') \geq \delta > 0$. In particular, the $\delta/2$-neighborhoods of $W$ and $W'$ are disjoint, i.e. $(W)_{\delta/2} \cap (W')_{\delta/2} = \emptyset$.

Choose $x \in (W)_{\delta/4} \setminus W$, so $x \notin X_-$. By construction of $X_-$, we get $d_{\Gamma}(x) \geq \epsilon$. Thus the ball $B$ of radius $\rho := \min\{ \epsilon/2, \delta/4 \}$ around $x$ projects injectively into $M = X\slash\Gamma$. In particular, the volumes of $B$ and $\pi(B)$ coincide. We repeat this procedure for every connected component $U$ of $M_-$ and thus get a ball $\pi(B) \subseteq M$ for every such $U$; observe that all these balls will be disjoint by choice of $x$ and $\rho$.

By the curvature assumptions, we have that the volume of an $r$-ball in $X$ is greater or equal to the volume $V(r,n)$ of an $r$-ball in $\R^n$. Thus also $\pi(B) \geq V(\rho,n)$ for each of the above balls $\pi(B)$ and so by disjointness, there can be at most $\Vol(M)/V(\rho,n)$ such balls. Hence $C := 1/V(\rho,n)$ fulfills the assumption. $C$ depends only on $\epsilon$ and $n$ because $\rho$ only depends on $\epsilon$ and $\delta(n)$.
\end{proof}

The thick part is compact and for dimension $n>2$ also connected.

\begin{lem}\label{Lem:DickerTeilKompakt}
$M_+$ is compact.
\end{lem}

\begin{proof}
Again, the proof is standard (see e.g. \cite{BenPet} Proposition D.2.6 for the case $X = \H^n$).

Choose a maximal $\epsilon$-discrete subset $\mathcal{M} \subseteq M_+$, i.e. $d(p_1,p_2) \geq \epsilon$ for all $p_1, p_2 \in \mathcal{M}$. Then in particular, the $(\epsilon/4)$-balls $B^M_{\epsilon/4}(p)$ around the points $p \in \mathcal{M}$ are disjoint. With similar arguments as in the previous proof, we get that the volume of each such ball is bounded from below by the volume $V(\epsilon/4,n)$ of an $\epsilon/4$-ball in $\R^n$. Again, we deduce $|\mathcal{M}| \leq \Vol(M)/V(\epsilon/4,n)$, which is finite. But the maximality of $\mathcal{M}$ yields
\[
M_+ \subseteq \bigcup\limits_{p \in \mathcal{M}} B^M_{2\epsilon}(p),
\]
so the closed subset $M_+ \subseteq M$ is also bounded and hence compact.
\end{proof}

\begin{lem}\label{Lem:DickerTeilZshgd}
$M_+$ is connected for $n>2$.
\end{lem}

\begin{proof}
Contracting the cusps won't change connectivity (see Lemma \ref{Lem:Spitzen}). Hence by Lemma \ref{Lem:Roehren} we just have to observe that $M_+$ might not be connected only if geodesics have codimension $\leq 1$, i.e. in dimension $n \leq 2$.
\end{proof}

We summarize the above results in the following theorem.

\begin{satz}\label{Satz:DickDuennZerlegungSichtbarkeit}
We have:
\begin{enumerate}
\item $M_+$ is a compact manifold with boundary.

\item $M_+$ is connected for dimension $n>2$.

\item The number of connected components of $M_-$ is bounded by $C \cdot \Vol(M)$, where $C = C(\epsilon,n) > 0$ is a constant only depending on $\epsilon$ and $n$.

\item The connected components $U$ of $M_-$ are of one of the following two shapes:
\begin{itemize}
\item Tubes (bounded components), i.e. $U$ is homeomorphic to a $D^{n-1}$-bundle over $\S^1$ and thus homotopy equivalent to $\S^1$.

\item Cusps (unbounded components), i.e. $U$ is homeomorphic to $V \times (0,\infty)$ for some compact $(n-1)$-dimensional manifold $V$. If $\partial U$ is the boundary of $U$ in $M$, then there is a strong deformation retraction of $U$ onto $\partial U$.
\end{itemize}
In particular, $M$ is homotopy equivalent to the compact manifold $M_C$ with boundary, which is constructed out of $M$ by contracting the cusps onto their common boundary with $M_+$. Equivalently, $M_C$ is the union of $M_+$ with the finitely many tubes. {\hfill$\square$}
\end{enumerate}
\end{satz}

\section{Efficient simplicial model}\label{Kapitel:Hauptresultat}

In this section we will prove the main result and its consequences, i.e. that every negatively curved visibility manifold of finite volume admits an efficient simplicial model for its thick part, which in turn yields bounds on the homology. The general idea of the proof is similar to the one in \cite{Sauer}: construct a good cover of the thick part (which, in general, will be larger than the thick part) that can be homotoped back onto the thick part. This ensures that the thick part will be homotopy equivalent to the good cover, which in turn is homotopy equivalent to its nerve complex via the nerve lemma; the nerve complex is the desired simplicial model.

The main difficulty lies in making the good cover stable under the flow from the thin to the thick part. In fact, we first have to say which flow we actually mean. \cite{Sauer} uses the fact that there is a nice, well-defined direction between the sublevel sets of commuting isometries to construct such a flow; suitably chosen balls (whose centers must not lie too close to the boundary of the thick part) will then be stable under this flow and can thus be taken to construct the good cover. While this works just fine in the pinched curvature case, a new problem came up in the $K<0$ visibility situation as we no longer had the virtual nilpotence of the parabolic stabilizer groups $\Gamma_z$ -- but this virtual nilpotence was crucial in order to have commuting isometries to construct the direction of the flow in the parabolic case, i.e. for cusps. For this reason we instead chose a different flow, namely the canonical one given by the geodesics to/from the parabolic fixed point. Although this at least gives us a flow to work with, making sure that sets are stable under this flow turns out to be much harder. Our approach is to force sets to be stable by cutting off the portion of the set which exceeds the thick part and instead continuing the remainder along the flow lines. This cut-off procedure in turn might lead to non-contractible sets or intersections, so the cover would no longer be good. Fortunately, this can be worked around: cutting the sets into smaller parts will finally yield a good cover.

Note that all these issues only arise near the cusps and we could actually take the same flow and construction as in \cite{Sauer} for the hyperbolic case, i.e. the tubes. Nonetheless, we replicate the new procedure also for the tubes, just to be consistent with the parabolic case. We further remark that in \cite{JiWu} it was recently proven that also in the visibility situation, the parabolic groups $\Gamma_z$ are in fact virtually nilpotent. Thus it might be the case that the strategy of \cite{Sauer} could as well be translated to the visibility setting. As the result of \cite{JiWu} was not available to us when we proved our statements, we will go on with our strategy regardless.

We retain the situation of the previous section, setting $\epsilon := \epsilon(n)/4$ and $\epsilon_{\gamma} := \epsilon$ for all $\gamma \in \Gamma \setminus \{ \id \}$. Again $M_+$ and $M_+$ denote the thick parts of $X$ and $M$, while $X_-$ and $M_-$ are the thin parts. As we chose constant levels $\epsilon_{\gamma} = \epsilon$, we get $X_+ = \{ x \in X : d_{\Gamma}(x) \geq \epsilon \}$. With $\delta = \delta(n)$ from Lemma \ref{Lem:AbstandKomponentenW}, we define the shrunken thick part
\[
M'_+ := X'_+\slash\Gamma = M \setminus (M_-)_{\delta/4}, \qquad\text{where}\qquad X'_+ := X \setminus (X_-)_{\delta/4}.
\]
By choice of $\delta$, the components of $(M_-)_{\delta/4}$ are in one-to-one correspondence with those of $M_-$ (and similarly for $X$). The covering sets of the good cover we construct later on will actually only cover the shrunken thick part (which will be homotopy equivalent to the regular thick part), but be contained in the regular thick part; this way, we still have control over the injectivity radius even though our cover is larger than the shrunken thick part.

\subsection{Defining the flow}

We will now construct the flow and show that the thick part is indeed homotopy equivalent to the shrunken thick part, where we first treat the situation in $X$ and then get the desired statements for $M$ via $\Gamma$-equivariance of the construction. Let $W$ be a component of $X_-$; the set $W' := (W)_{\delta/4} \setminus W$ is the corresponding component of $X_+ \setminus X'_+$ that we want to retract onto $X_+$. Denote the set of all those $W'$ by $\mathcal{W}$.

In the parabolic case -- i.e. $\pi(W)$ is a cusp of $M$ -- let $z \in X(\infty)$ be the parabolic fixed point; recall that $\Gamma_W = \Gamma_z$. For every $x \in \partial W$, let $c_x$ denote the unique unit speed geodesic from $z$ through $x$, with parametrization $c_x(0) = x$ and $c_x(-\infty) = z$ (i.e. flowing away from $z$). By monotonicity of $t \mapsto d_{\gamma}(c_x(t))$ for every nontrivial $\gamma \in \Gamma_z$, we see that $c_x(t)$ is in the interior of $W$ for all $t<0$ and outside the closure of $W$ for all $t>0$. Moreover, for distinct $x, x' \in \partial W$, the corresponding $c_x, c_{x'}$ are disjoint. Thus for all $y \in W'$, we have a representation as $y = c_x(t_y)$ for unique $x \in \partial W$ and $t_y > 0$. Let $T_x$ denote the time where $c_x$ enters $X'_+$ for the first time.

\begin{lem}\label{Lem:Stetigkeit_Eintrittszeit}
The entry time
\[
T: \partial W \rightarrow (0,\infty), \quad x \mapsto T_x
\]
is well-defined and continuous.
\end{lem}

\begin{proof}
We first show that $T$ is well-defined. Let $y' \in \partial X'_+$, so there is $y \in \partial X_-$ with $d(y',y) = \delta/4$. Since $d_{\Gamma}(y) = \epsilon$ and by the choice of $\epsilon$  and $\delta$, we get $d_{\Gamma}(y') < \epsilon(n)/2$. By Lemma \ref{Lem:Eberlein3.1ef}, also $d_{\Gamma}(c_x(t)) > \epsilon(n)/2$ for some $t \in \R$; since $d_{\Gamma}(c_x(0)) = \epsilon(n)/4$, the existence of an intersection of $c_x$ with $X'_+$ follows via the intermediate value theorem. Hence $T_x$ is well-defined.

{
\begin{figure}
\centering

\begin{tikzpicture}[scale=0.75]


\begin{scope}

\clip[xshift=-3cm, yshift=+2cm] (0.25,-3) -- (7.5,-3) -- (7.5,4.5) -- (0.25,4.5) -- cycle;


\fill[color=gray!0, xshift=-3cm, yshift=+2cm] (7.5,-3) -- (8.75, 4.5) -- (0.75, 4.5) -- (-0.5,-3) -- cycle; 


\fill[color=gray!8, xshift=-3cm, yshift=+2cm] (0.25,1.5) arc (160:30:2.75cm and 1.2cm) .. controls (6.75, 2) .. (8.25,1.5) -- (8.75, 4.5) -- (0.25, 4.5) -- cycle; 


\fill[color=gray!8, xshift=-3cm, yshift=+2cm] (-0.375,-2.25) arc (160:30:2.75cm and 1.2cm) .. controls (6.125, -1.75) .. (7.625,-2.25) -- (7.5, -3) -- (-0.5, -3) -- cycle; 


\draw[ultra thin, xshift=0 cm, yshift=0 cm,decoration={markings, mark=at position 0.75 with {\arrow{<}}},
        postaction={decorate}] (1.45,6.5) .. controls (-0.3,3.03) and (-0.7,2.03) .. (-1,0.53);
\draw[ultra thin, xshift=0 cm, yshift=0 cm,decoration={markings, mark=at position 0.75 with {\arrow{<}}},
        postaction={decorate}] (0.95,6.5) .. controls (-0.8,3) and (-1.2,2) .. (-1.5,0.5);


\coordinate (x) at (-1,0.53); 
\coordinate (x') at (-1.5,0.5); 
\coordinate (x_1) at (-0.18,3.1); 
\coordinate (x'_1) at (-0.52,3.45); 
\coordinate (cxTx) at (0.35,4.26); 
\coordinate (cx'Tx') at (-0.13,4.29); 
\coordinate (x_2) at (0.96,5.52); 
\coordinate (x'_2) at (0.62,5.86); 

\draw[ultra thin, xshift=-3cm, yshift=+2cm] (0.25,1.5) arc (160:30:2.75cm and 1.2cm) .. controls (6.75, 2) .. (8.25,1.5); 

\draw[ultra thin, xshift=-3cm, yshift=+2cm] (-0.375,-2.25) arc (160:30:2.75cm and 1.2cm) .. controls (6.125, -1.75) .. (7.625,-2.25); 

\draw[ultra thin, fill=gray, fill opacity=0.1] (cxTx) circle (1.5cm); 
\draw[ultra thin, fill=gray, fill opacity=0.1] (x_1) circle (0.7cm); 
\draw[ultra thin, fill=gray, fill opacity=0.1] (x_2) circle (0.7cm); 

\filldraw (x) circle (0.75pt) node[below, xshift=0pt, yshift=-1pt] {$x$}; 
\filldraw (x') circle (0.75pt) node[below, xshift=0pt, yshift=2pt] {$x'$}; 
\filldraw (x_1) circle (0.75pt) node[right, yshift=-1pt] {$x_1$}; 
\filldraw (x'_1) circle (0.75pt) node[above left, xshift=1pt, yshift=-3pt] {$x'_1$}; 
\filldraw (cxTx) circle (0.75pt) node[below right, xshift=-2pt, yshift=3pt] {$c_x(T_x)$}; 
\filldraw (cx'Tx') circle (0.75pt) node[above left, xshift=2pt, yshift=-1pt] {$c_{x'}(T_{x'})$}; 
\filldraw (x_2) circle (0.75pt) node[right, yshift=-1pt] {$x_2$}; 
\filldraw (x'_2) circle (0.75pt) node[above left, xshift=0pt, yshift=-2pt] {$x'_2$}; 

\node at (3.25,5) {$X'_+$}; 
\node at (3.25,1.75) {$W'$}; 
\node at (3.25,-0.5) {$W$}; 

\node at (-0.4,1.58) {$c_x$}; 
\node at (-1.6,1.6) {$c_{x'}$}; 

\end{scope}

\end{tikzpicture}

\caption{
Situation in the proof of Lemma \ref{Lem:Stetigkeit_Eintrittszeit}. The balls around $x_1$ and $x_2$ have radius $\delta_1 < \epsilon_0/4$ and $\delta_2 < \epsilon_0/4$ respectively, while the ball around $c_x(T_x)$ is of radius $< \epsilon_0/2$. Moreover, the distance between $x'$ and $x$ is at most $\epsilon_0/4$.
}
\label{Bild_Stetigkeit_Eintrittszeit}

\end{figure}
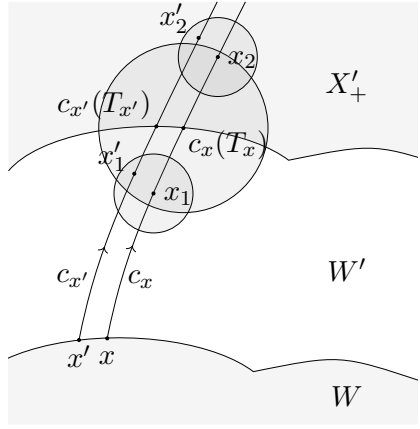

}

Let $x \in \partial W$ and $\epsilon_0 > 0$; in order to prove continuity in $x$, we have to find $\delta_0 > 0$ such that $|T_x - T_y| < \epsilon_0$ for all $y \in \partial W$ with $d(x,y) < \delta_0$. Figure \ref{Bild_Stetigkeit_Eintrittszeit} will help to illustrate the situation. Let $x_1$ be a point on $c_x$ such that $x_1 \in W' \setminus \overline{W}$ and $c_x$ enters $X'_+$ less than $\epsilon_0/2$ units of time after passing through $x_1$, i.e. $x_1 = c_x(t_{x_1})$ for some $t_{x_1} > T_x - \epsilon_0/2$. Since $W' \setminus \overline{W}$ is open, there is $\delta_1 > 0$ sucht that $B_{\delta_1}^X(x_1) \subseteq W' \setminus \overline{W}$; note that $\delta_1 \leq \epsilon_0/2$. Let $x' \in \partial W \cap B_{\epsilon_0/4}^X(x)$ be another point with existing intersection $x'_1 \in c_{x'}([0,\infty)) \cap B_{\delta_1}^X(x_1) \neq \emptyset$; similarly, $x'_1 = c_{x_1}(t_{x'_1})$ for some $t_{x'_1} < T_{x'}$. We will now assume $\delta_1 < \epsilon_0/4$ without restriction. Using $t_{x_1} = d(x, x_1) \leq d(x, x') + d(x', x'_1) + d(x'_1, x_1)$, we deduce
\[
t_{x'_1} = d(x', x'_1) \geq t_{x_1} - d(x, x') - d(x'_1, x_1) > t_{x_1} - \epsilon_0/2.
\]
Together with $t_{x_1} > T_x - \epsilon_0/2$ and $t_{x'_1} < T_{x'}$ this yields
\begin{equation}\label{Gleichung:Beweis_Stetigkeit_Eintrittszeit_1}
T_{x'} > t_{x'_1} > t_{x_1} - \epsilon_0/2 > T_x - \epsilon_0/2 - \epsilon_0/2 = T_x - \epsilon_0.
\end{equation}
On the other hand, $c_x((T_x, T_x + t))$ is in the interior of $X'_+$ for sufficiently small $t>0$. Hence there is some $x_2$ on $c_x$ in the interior of $X'_+$ such that $x_2$ is reached at most $\epsilon_0/2$ units of time after $c_x$ enters $X'_+$ (where we assume that, without restriction, $\epsilon_0$ is sufficiently small w.r.t. $t$), i.e. $x_2 = c_x(t_{x_2})$ for some $t_{x_2} < T_x + \epsilon_0/2$. Again, there is $\delta_2 > 0$ with $B_{\delta_2}^X(x_2) \subseteq X'_+$, where $\delta_2 < \epsilon_0/4$ without restriction. If again $x' \in \partial W \cap B_{\epsilon_0/4}^X(x)$ is a point with existing intersection $x'_2 \in c_{x'}([0,\infty)) \cap B_{\delta_2}^X(x_2) \neq \emptyset$, we have $x'_2 = c_{x'}(t_{x'_2})$ for some $t_{x'_2} > T_{x'}$. Similar to the previous argument, we deduce
\begin{equation}\label{Gleichung:Beweis_Stetigkeit_Eintrittszeit_2}
T_{x'} < t_{x'_2} < t_{x_2} + \epsilon_0/2 < T_x + \epsilon_0/2 + \epsilon_0/2 = T_x + \epsilon_0.
\end{equation}
Recall that
\[
c: \partial W \times \R \rightarrow X, \quad (x,t) \mapsto c(x,t) := c_x(t).
\]
is continuous. Taking the preimage of $B_{\delta_1}^X(x_1)$ under $c$ and projecting onto the factor $\partial W$ yields an open neighborhood $U_x^{(1)} \subseteq \partial W$ of $x$ in $\partial W$, such that for all $x' \in U_x^{(1)}$, the flow along $c_{x'}$ intersects the ball $B_{\delta_1}^X(x_1)$. Assuming $B_{\epsilon_0/4}^X(x) \subseteq U_x^{(1)}$ without restriction, we deduce $T_{x'} > T_x - \epsilon_0$ by inequality (\ref{Gleichung:Beweis_Stetigkeit_Eintrittszeit_1}). Similarly, using $B_{\delta_2}^X(x_2)$, we get an open neighborhood $U_x^{(2)} \subseteq \partial W$ of $x$ in $\partial W$, such that by inequality (\ref{Gleichung:Beweis_Stetigkeit_Eintrittszeit_2}), we have $T_{x'} < T_x + \epsilon_0$ for all $x' \in U_x^{(2)}$. In summary, there is an open ball $B \subseteq U_x^{(1)} \cap U_x^{(2)}$ around $x$ in $\partial W$ such that
\[
|T_x - T_{x'}| < \epsilon_0
\]
for all $x' \in B$. Now the radius $\delta_0$ of $B$ is the desired $\delta_0$ for the continuity of $T$ in $x$.
\end{proof}

In particular, we have $W' = \bigcup_{x \in \partial W} c_x( [0,T_x) )$, where the union is disjoint. Define the flow $F_W: X_+ \times [0,1] \rightarrow X_+$ via
\[
F_W(y,t) = \begin{cases}
c_x \big( (1-t)\cdot t_y + t \cdot T_x \big) & \text{if } y = c_x(t_y) \in c_x( [0,T_x) ) \subseteq W',\\
y & \text{else}.
\end{cases}
\]
By the above, this is continuous. Note that $\gamma c_x = c_{\gamma x}$ for $\gamma \in \Gamma_z$, and similarly $T_{\gamma x} = T_x$. With this we see that $F_W$ is $\Gamma_z$-equivariant, i.e. $F_W(\gamma y, t) = \gamma F_W(y,t)$ for $\gamma \in \Gamma_z$. Moreover, $F_W$ clearly defines a strong deformation retraction of $X_+$ onto $X_+ \setminus W'$, which for time $t=1$ induces a homeomorphism $F_W(\cdot, 1)_{|\partial W}$ between $\partial W$ and the corresponding boundary component $\overline{W'} \cap (X_+ \setminus W')$ of $X_+ \setminus W'$.

For the hyperbolic case -- i.e. $\pi(W)$ is a tube of $M$ -- let $A \subseteq W$ be the unique axis of the elements of $\Gamma_W$. Denoting the projection geodesic from $x \in \partial W$ to $\pi_A(x) \in A$ by $c_x$, where $c_x(0) = x$ and $c_x(-d(x, \pi_A(x))) = \pi_A(x)$, we can repeat the above constructions, namely defining the entry time $T_x$ and the corresponding flow $F_W$. Again, $F_W$ will be a strong deformation retraction of $X_+$ onto $X_+ \setminus W'$, inducing a homeomorphism between $\partial W$ and the respective boundary component $\overline{W'} \cap (X_+ \setminus W')$ of $X_+ \setminus W'$ for time $t=1$; the arguments are completely analogous to the parabolic case.

Having defined a suitable flow $F_W$ for every component $W$ of $X_-$ separately, we summarize the above constructions in the general setting:

\begin{lem}\label{Lem:DefinitionEigenschaftenF}
The map $F: X_+ \times [0,1] \rightarrow X_+$, given by
\[
F(y,t) := \begin{cases}
F_W(y,t) & \text{if } y \in W' \in \mathcal{W},\\
y & \text{else},
\end{cases}
\]
is a well-defined strong deformation retraction of $X_+$ onto $X'_+$. $F$ is $\Gamma$-equivariant and for time $t=1$ induces a homeomorphism $F(\cdot,1)_{|\partial X_+}$ between $\partial X_+$ and $\partial X'_+$.
\end{lem}

\begin{proof}
Recall that by choice of $\delta$, the different $W' \in \mathcal{W}$ don't intersect, so $F$ is indeed well-defined. As every $F_W$ retracted onto $X_+ \setminus W'$, the map $F$ retracts onto $X'_+ = X_+ \setminus \bigcup_{W' \in \mathcal{W}} W'$. Similarly, the homeomorphism statement for $F(\cdot, 1)_{|\partial X_+}$ follows. Note that the $\Gamma$-equivariance of $F$ can be deduced from the $\Gamma_W$-equivariance of the $F_W$'s together with the fact that isometries $\gamma \in \Gamma \setminus \Gamma_W$ interchange $W$ with another component $V$ of $X_-$.
\end{proof}

Since $F$ is $\Gamma$-equivariant, it descends to a similar flow $f$ in $M$.

\begin{lem}\label{Lem:Sauer3.13Analogon}
The map $f: M_+ \times [0,1] \rightarrow M_+$ is a strong deformation retraction of $M_+$ onto $M'_+$, which for time $t=1$ induces a homeomorphism $f(\cdot,1)_{|\partial M_+}$ between $\partial M_+$ and $\partial M'_+$. {\hfill$\square$}
\end{lem}

\subsection{Stable covering sets}

Our aim is to construct a good cover of the (shrunken) thick part $M'_+$. For the covered subspace to be homotopy equivalent to $M'_+$, we need it to be stable under the flow $f$ retracting $M_+$ onto $M'_+$; this is because the covering sets will in general exceed $M'_+$ and we thus need to push them back. By \textit{stable under the flow} we simply mean that if the flow $f$ enters a covering set, it will remain inside the cover until it reaches $\partial M'_+$, where we stop flowing. To achieve stability of a covering set, we will simply cut of its part outside $M'_+$ and then extend the remaining set along the flow lines given by $f$. Note that equivalently, we could treat the situation in $X$ with flow $F$, as long as the construction remains $\Gamma$-equivariant.

More precisely, let $B_r^M(p)$ be a ball of radius $r$ in $M$ around $p \in M'_+$. Here,
\[
r := \delta(n)/4
\]
is fixed; note that by choice of $r$, the ball $B_r^M(p)$ is convex and intersects at most one boundary component of $M'_+$. If $B_r^M(p) \subseteq M'_+$, we won't change $B_r^M(p)$. On the other hand, if $B_r^M(p) \setminus M'_+ \neq \emptyset$, then $A := B_r^M(p) \cap \partial M'_+$ is a nonempty set, which is entirely contained in a unique component of $\partial M'_+$. Let $\pi(W')$ (for $W' \in \mathcal{W}$) and $W$ denote the corresponding components of $M_+ \setminus M'_+$ and $X_-$, respectively, i.e. $B_r^M(p) \cap \pi(W') \neq \emptyset$.

The geodesics $c_x$ of the points $x \in A' := \{ x \in \partial W \,|\, \pi(c_x(T_x)) \in A \}$ all meet points of $\partial X'_+$ at entry time $T_x$, which -- after projecting to $M$ -- correspond to points in $A$. Let
\[
A_S := \pi\Big( \bigcup\limits_{x \in A'} c_x\big( (T_x/2, T_x) \big) \Big),
\]
which is what we previously meant by \textit{extension along the flow lines}. The final covering set is now obtained by setting
\[
B' := \big( B_r^M(p) \cap M'_+ \big) \cup A_S.
\]
Hence $B'$ is an open set which is stable under the flow $f$, i.e. $f(y,t) \in B'$ for all $t \geq t_0$, if $f(y,t_0) \in B'$. Figure \ref{Bild_Konstruktion_Ueberdeckungsmengen} illustrates the construction.

{
\begin{figure}
\centering

\begin{tikzpicture}[scale=0.75]


\begin{scope}[xshift=-6cm]

\clip[xshift=-3cm, yshift=+2cm] (0,-3) -- (7.5,-3) -- (7.5,3) -- (0,3) -- cycle;


\fill[color=gray!0, xshift=-3cm, yshift=+2cm] (7.5,-3) -- (8.5, 3) -- (0.5, 3) -- (-0.5,-3) -- cycle; 


\foreach \s in {1, 2, 3, 4, 5, 6, 7, 8, 9, 10, 11, 12, 13}
      \draw[ultra thin, xshift=0 cm, yshift=0 cm,decoration={markings, mark=at position 0.7 with {\arrow{<}}},
        postaction={decorate}] (-2+0.3*\s,3) .. controls (-2.3+0.3*\s,2.5) and (-2.7+0.3*\s,2) .. (-3.2+0.3*\s,0);


\fill[color=gray!8, xshift=-3cm, yshift=+2cm] (0,0) arc (160:30:2.75cm and 1.2cm) .. controls (6.5, 0.5) .. (8,0) -- (8.5, 3) -- (0, 3) -- cycle; 


\fill[color=gray!8, xshift=-3cm, yshift=+2cm] (-0.35,-2.1) arc (160:30:2.75cm and 1.2cm) .. controls (6.15, -1.6) .. (7.65,-2.1) -- (7.5, -3) -- (-0.5, -3) -- cycle; 


\coordinate (p) at (-0.13,2.95); 

\draw[ultra thin, xshift=-3cm, yshift=+2cm] (0,0) arc (160:30:2.75cm and 1.2cm) .. controls (6.5, 0.5) .. (8,0); 

\draw[ultra thin, xshift=-3cm, yshift=+2cm] (-0.35,-2.1) arc (160:30:2.75cm and 1.2cm) .. controls (6.15, -1.6) .. (7.65,-2.1); 

\draw[very thin, fill=gray, fill opacity=0.1] (p) circle (1.5cm); 

\filldraw (p) circle (0.75pt) node[above left, xshift=1pt, yshift=-3pt] {$p$}; 

\node at (3.625,4.25) {$M'_+$}; 
\node at (3.625,1.25) {$M_+$}; 
\node at (3.625,-0.5) {$M_-$}; 

\node at (-0.1, 4) {$B_r^M(p)$};

\end{scope}


\begin{scope}[xshift=3cm]

\clip[xshift=-3cm, yshift=+2cm] (0,-3) -- (7.5,-3) -- (7.5,3) -- (0,3) -- cycle;


\fill[color=gray!0, xshift=-3cm, yshift=+2cm] (7.5,-3) -- (8.5, 3) -- (0.5, 3) -- (-0.5,-3) -- cycle; 


\foreach \s in {1, 2, 3, 4, 5, 6, 7, 8, 9, 10, 11, 12, 13}
      \draw[ultra thin, xshift=0 cm, yshift=0 cm,decoration={markings, mark=at position 0.7 with {\arrow{<}}},
        postaction={decorate}] (-2+0.3*\s,3) .. controls (-2.3+0.3*\s,2.5) and (-2.7+0.3*\s,2) .. (-3.2+0.3*\s,0);


\fill[color=gray!8, xshift=-3cm, yshift=+2cm] (0,0) arc (160:30:2.75cm and 1.2cm) .. controls (6.5, 0.5) .. (8,0) -- (8.5, 3) -- (0, 3) -- cycle; 


\fill[color=gray!8, xshift=-3cm, yshift=+2cm] (-0.35,-2.1) arc (160:30:2.75cm and 1.2cm) .. controls (6.15, -1.6) .. (7.65,-2.1) -- (7.5, -3) -- (-0.5, -3) -- cycle; 


\coordinate (p) at (-0.13,2.95); 

\draw[ultra thin, xshift=-3cm, yshift=+2cm] (0,0) arc (160:30:2.75cm and 1.2cm) .. controls (6.5, 0.5) .. (8,0); 

\draw[ultra thin, xshift=-3cm, yshift=+2cm] (-0.35,-2.1) arc (160:30:2.75cm and 1.2cm) .. controls (6.15, -1.6) .. (7.65,-2.1); 

\draw[very thin, fill=gray, fill opacity = 0.1] (-2.15,1.5) arc (130:52.5:2.4cm and 1.2cm) .. controls (1.1,2.12) .. (1.31,2.52) arc (-17:191:1.5cm) .. controls (-1.9,2.15) and (-1.98,1.98) .. cycle; 


\filldraw (p) circle (0.75pt) node[above left, xshift=1pt, yshift=-3pt] {$p$}; 

\node at (3.625,4.25) {$M'_+$}; 
\node at (3.625,1.25) {$M_+$}; 
\node at (3.625,-0.5) {$M_-$}; 

\node at (-0.2, 4) {$B'$};

\end{scope}

\end{tikzpicture}

\caption{
The stable covering sets are constructed out of balls by cutting off the part outside $M'_+$ and extending the remaining part along the flow lines.
}
\label{Bild_Konstruktion_Ueberdeckungsmengen}

\end{figure}
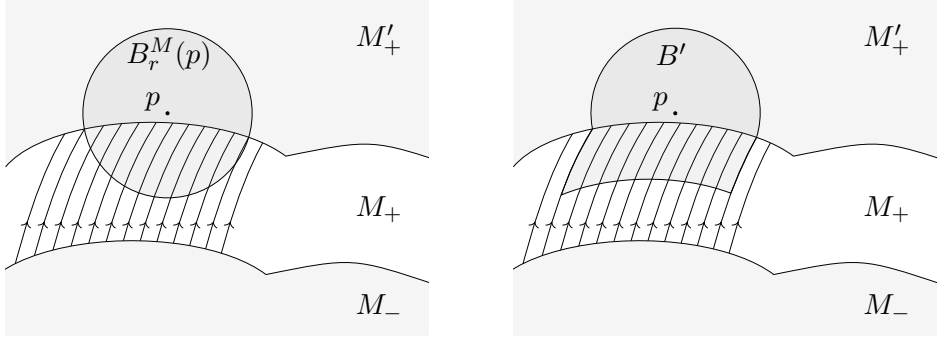
}

\subsection{Cutting and refining}

As mentioned previously, cutting off the part of a ball $B_r^M(p)$ outside $M'_+$ might lead to the resulting stable covering set $B'$ (or rather their intersections) no longer being contractible. In this section, we will show how to deal with this issue. Recall that by our choice of radius $r$, the initial balls $B_r^M(p)$ are isometric to their lifts $B_r^X(x)$ in $X$, where $\pi(x) = p$ (and $p \in M'_+$); moreover, a stable covering set $B'$ is homotopy equivalent to its part in $M'_+$ (and similarly for their lifts in $X$). Thus in order to study (the homotopy type of) a stable covering set $B'$, we can equivalently look at the set $B_r^X(x) \cap X'_+$ (and similarly for intersections) and will also denote the latter set by $B'$, for simplicity.

Note that we have $(X_-)_{\delta/4} = \bigcup_{\gamma \in \Gamma \setminus\{\id\}} ( \{ d_{\gamma} < \epsilon \} )_{\delta/4}$ and thus
\[
B' = B_r^X(x) \setminus \bigcup\limits_{\gamma \in \Gamma \setminus\{\id\}} ( \{ d_{\gamma} < \epsilon \} )_{\delta/4},
\]
since $X'_+ = X \setminus (X_-)_{\delta/4}$. The following lemma serves as a first step to studying the sets in the above equation, as well as their relationship.

\begin{lem}\label{Lem:Eigenschaften_Konstruktion_B'}
We have:
\begin{enumerate}
\item The sets $( \{ d_{\gamma} < \epsilon \} )_{\delta/4}$ are convex.

\item There is a constant $\kappa = \kappa(n) \in \N$ only depending on $n$, such that the number of sets $( \{ d_{\gamma} < \epsilon \} )_{\delta/4}$ intersecting $B_r^X(x)$ nontrivially is at most $\kappa$.

\item For all $x \in \partial W$, the function $[0,\infty) \rightarrow \R_{\geq 0}$, $t \mapsto d(c_x(t), \{ d_{\gamma} < \epsilon \} )$ (where $\gamma \in \Gamma_W \setminus \{\id\}$) is strictly increasing in $t$ (again, $c_x$ denotes a geodesic of the flow $F$).

\item If $c_x(t_0) \in ( \{ d_{\gamma} < \epsilon \} )_{\delta/4}$ (where $\gamma \in \Gamma_W \setminus \{\id\}$) for some $t_0 \in \R$ and $x \in \partial W$, then $c_x(t) \in {( \{ d_{\gamma} < \epsilon \} )_{\delta/4}}$ for all $t \leq t_0$.

\item For all $x \in \partial W$ the function $[0,\infty) \rightarrow \R_{\geq 0}$, $t \mapsto d(c_x(t), W)$ is strictly increasing in $t$.
\end{enumerate}
\end{lem}

\begin{proof}
\begin{enumerate}
\item Observe that $( \{ d_{\gamma} < \epsilon \} )_{\delta/4} = ( \{ d_{\gamma} \leq \epsilon \} )_{\delta/4}$, and the latter set can be seen as a sublevel set of the convex (see \cite{BGS} Section 1.6 Exercise (iii)) function $d(\cdot, \{ d_{\gamma} \leq \epsilon \})$.

\item Let $\gamma \in \Gamma \setminus \{ \id \}$ such that $B_r^X(x) \cap ( \{ d_{\gamma} < \epsilon \} )_{\delta/4} \neq \emptyset$, so $d(y,y')$ for some $y \in B_r^X(x)$ and $y' \in \{ d_{\gamma} < \epsilon \}$. Hence
\[
d_{\gamma}(x) \leq d_{\gamma}(y') + 2 d(x,y') < \epsilon + 2r + \delta/2 < \epsilon(n)
\]
by choice of $\epsilon$, $r$ and $\delta$. Since $x \in X'_+ \subseteq X_+$ (i.e. $d_{\Gamma}(x) \geq \epsilon$), we see that there can be at most $N := N(n,\epsilon,\epsilon(n))$ such $\gamma$ (see \cite{Sauer} Lemma 3.2 for the definition of $N$). As $N$ depends only on $n$, we conclude that $\kappa(n) := \lceil N \rceil$ satisfies the statement.

\item First, assume further that $x \in \partial \{ d_{\gamma} < \epsilon \}$. In 1. we saw that for $D := \{ d_{\gamma} < \epsilon \}$, the distance function $d_D$ is convex. By our parametrization we have $c_x(0) = x \in \partial D$, so $d_D(c_x(0)) = 0$. As $t \mapsto d_{\gamma}(c_x(t))$ is strictly increasing, also $d_D(c_x(t)) > 0$ for all $t > 0$. Hence $t \mapsto d_D(c_x(t))$ is a convex function with strict minimum in $0$, thus strictly increasing.

This already proves the statement in the case of tubes (i.e. $\pi(W)$ is a tube in $M$), since $c_x$ always intersects $\partial D$ in that situation.

In the following, we will thus assume that $W$ projects onto a cusp; let $z \in X(\infty)$ be the parabolic fixed point. Moreover, we're only left with treating the situation where $c_x$ and $D$ do not intersect\footnote{This can not happen in the pinched curvature case. There, $d_{\gamma}(c_x(t)) \rightarrow 0$ as $t \rightarrow -\infty$ for all $\gamma \in \Gamma_W \setminus \{ \id \}$, so $c_x$ and $D$ would always intersect.}.

Let $x'$ be the projection point of $x$ onto $\overline{D} = \{ d_{\gamma} \leq \epsilon \}$ (so $d(x, D) = d(x,x')$) and $HB$ be a horoball around $z$, not containing either of $x$ and $x'$; denote the projection onto $HB$ by $\pi_{HB}$. Note that $\pi_{HB}(x') \in D$ and $\pi_{HB}(x) = c_x(t_0)$ for some $t_0 < 0$. As the projection is (strictly) distance decreasing (see \cite{BO'N} Proposition 3.4), we obtain
\begin{align*}
d(c_x(t_0), D) &\leq d(c_x(t_0), \pi_{HB}(x')) = d(\pi_{HB}(x), \pi_{HB}(x')) \\
&< d(x, x') = d(c_x(0), D).
\end{align*}
In particular, the convex function $t \mapsto d(c_x(t), D)$ defined on all of $\R$ is not constant (and thus also has no local/global maximum). If there was a local minimum in $t_1 \in \R$, then by convexity, it would be global. But the argument analogous to the above inequality (with $c_x(t_1)$ instead of $c_x(0) = x$) would show that the projection of $c_x(t_1)$ onto a suitable horoball $HB'$ around $z$ had a smaller distance to $D$ than $c_x(t_1)$, contradicting the assumption that $c_x(t_1)$ was a global minimum of $t \mapsto d(c_x(t), D)$.

Consequently, the convex function $t \mapsto d(c_x(t), D)$ defined on all of $\R$ has no (local or global) extrema. In view of the above inequality, the only possible scenario is thus that $t \mapsto d(c_x(t), D)$ is strictly increasing, which then also holds for the restriction to $[0,\infty)$.

\item This is a consequence of 3.

\item As $W = \bigcup_{\gamma \in \Gamma_W \setminus \{\id\}} \{ d_{\gamma} < \epsilon \}$, this also follows from 3.
\end{enumerate}
\end{proof}

So the above lemma tells us that $B'$ is obtained from $B_r^X(x)$ by removing convex subsets of it, where the number of convex subsets removed is bounded by $\kappa = \kappa(n)$. Still, how the convex subsets lie inside the ball (or an intersection of several balls) could be quite complicated, resulting in a complicated $B'$. We thus have to make sure that the situation still is well-behaved; this is done in the following lemma.

\begin{lem}\label{Lem:Durchschnitt_Kugeln_Ausschneidung_homoeomorph_Halbkugel}
Let $B_i := B_r^X(x_i)$ ($i=1,\ldots,k$) be balls as in the present situation, where $\bigcap_{i=1}^k B_i \setminus ( \{ d_{\gamma} < \epsilon \} )_{\delta/4} \neq \emptyset$. Then
\[
\Big( \bigcap\limits_{i=1}^k B_i , \bigcap\limits_{i=1}^k B_i \cap ( \{ d_{\gamma} < \epsilon \} )_{\delta/4} \Big) \cong (D^n, D^n_+)
\]
are homeomorphic as pairs. Here $D^n$ denotes the open unit ball in $\R^n$ and $D^n_+ := \{ x \in D^n : x_1 > 0 \}$ half a ball.
\end{lem}

\begin{proof}
Some intuition for the situation can be gained from Figure \ref{Bild_Kleine_Kugel_grosse_Kugel}.

{
\begin{figure}
\centering

\begin{tikzpicture}[scale=0.75]



\fill[color=gray!2, xshift=0cm, yshift=0cm] (0,0) .. controls (1.25,2.5) and (3.25, 3.25) .. (4,3.5) -- (4, 0) -- cycle; 



\fill[color=gray!8, xshift=0cm, yshift=0cm] (0,0) .. controls (1.25,2.5) and (3.25, 3.25) .. (4,3.5) -- (4, 6) -- (-3, 6) -- (-3, 0) -- cycle; 


\coordinate (x) at (-1,3); 


\draw[ultra thin, xshift=0cm, yshift=0cm] (0,0) .. controls (1.25,2.5) and (3.25, 3.25) .. (4,3.5); 

\draw[ultra thin, fill=gray, fill opacity=0.1] (x) circle (1.5cm); 

\draw[very thin, dashed, xshift=0cm, yshift=0cm] (-2,0) .. controls (-0.5,3.5) and (1.5, 4.5) .. (4,5.5); 

\draw [decorate,decoration={brace,amplitude=5pt},xshift=0pt,yshift=0pt]
(1.95,4.55) -- (2.9,3.05) node [black,midway,xshift=12pt, yshift=5pt] 
{\scriptsize $\delta/4$};

\filldraw (x) circle (0.75pt) node[below left, xshift=1pt, yshift=1pt] {$x$}; 

\node at (1.25,5.25) {$X_+$}; 
\node at (2.5,0.7) {$\{ d_{\gamma} < \epsilon \}$}; 

\node at (-1, 3.9) {$B_r^X(x)$};

\end{tikzpicture}

\caption{If for a ball "of small radius" (here: $B_r^X(x)$ with radius $r$), we remove a distant piece "of large radius" (here: the thickening of $\{ d_{\gamma} < \epsilon \}$ by $\delta/4 \geq r$), there won't be any new holes cut into it.}
\label{Bild_Kleine_Kugel_grosse_Kugel}

\end{figure}
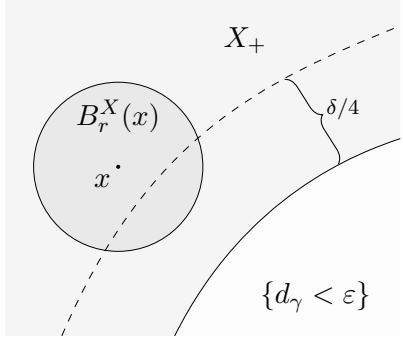
}

Let us first assume that $k=1$, i.e. $B := B_r^X(x_1)$, and write $C := \{ d_{\gamma} < \epsilon \}$. Using the presentations $(C)_t = \bigcup_{y \in C} B_t^X(y)$ and similarly $\overline{(C)_{t}} = \bigcup_{y \in \overline{C}} \overline{B_t^X(y)}$ (for arbitrary $t > 0$), we see that for every $y_0 \in \partial (C)_{t}$ there is a point $y_1 \in \partial C$ such that $y_0 \in \partial B_t^X(y_1)$. Moreover, this point $y_1$ is unique: we have $d(y_0, \overline{C}) = t$ and the existence of different $y_1, y'_1 \in \partial C$ with $y_0 \in \partial B_t^X(y_1)$, $y_0 \in \partial B_t^X(y'_1)$ -- i.e. $d(y_0, y_1) = t = d(y_0, y'_1)$ -- would contradict the uniqueness of the projection point of $y_0$ onto the closed, convex set $\overline{C} = \{ d_{\gamma} \leq \epsilon \}$.

As $x_1 \in X'_+$, we have $d(x_1,C) \geq d(x_1, X_-) \geq \delta/4$ and thus there exists a well-defined $t_0 \in [0,\delta/4)$ such that $\overline{B} \cap \overline{(C)_{t_0}}$ is nonempty for the first time; moreover, for this $t_0$, that intersection will consist of precisely one point $x_0$, because it would contradict convexity if there was also a point other than $x_0$. We will call $x_0$ the first contact point of $B$ with the thickenings of $C$.

In view of the monotonicity of $d(\cdot, C)$ along the flow lines (see Lemma \ref{Lem:Eigenschaften_Konstruktion_B'}) and the convexity of $B \cap (C)_{\delta/4}$, it simply remains to show that $(C)_{\delta/4}$ does not meet the upper side of $\partial B$ at a point which is not contained in the same connected component (which itself is homeomorphic to $D^{n-1}$) of $\partial B \cap (C)_{\delta/4}$ as $x_0$. By upper side of $\partial B$ we mean the points that -- when flowing away from the parabolic fixed point or the hyperbolic axis, respectively -- only arise as the second intersection point of $c_x$ with $\partial B$.

Assume to the contrary that there was another intersection of the thickenings of $C$ with the upper side of $B$. Then similar to the definition of $x_0$ we would have some $t' < \delta/4$ such that in that area, the intersection $\overline{B} \cap \overline{(C)_{t'}}$ would consist of precisely one point $x' \in \partial B$. By the above arguments, we find a unique $y' \in \partial C$ with $x' \in \partial B_{t'}^X(y')$; moreover, the boundary of that ball (and thus the boundary of $(C)_{t'}$) is tangent to $\partial B$ at that point $x'$, i.e. $T_{x'} \partial B = T_{x'} \partial B_{t'}^X(y') = T_{x'} \partial (C)_{t'}$ as subspaces of $T_{x'} X$. By the Gauss lemma, the radii of the balls $B$ and $B_{t'}^X(y')$ are orthogonal to their respective boundary in $x'$, hence their directions have to coincide. Thus the radius of $B_{t'}^X(y')$ for $x'$ (i.e. the geodesic piece $c_{y',x'}$ from $y'$ to $x'$) is a subset of the radius of $B$ for $x'$ (i.e. the geodesic piece $c_{x_1,x'}$ from $x_1$ to $x'$), or the other way round. In the first case, we would have
\[
y' = c_{y',x'}(0) \in c_{y',x'}([0,t')) \subseteq c_{x_1,x'}([0,r))
\]
and consequently $d(y', x_1) < r$; but because $y' \in \partial C$ and $r \leq \delta/4$, this would contradict $x_1 \notin (X_-)_{\delta/4}$. Similarly, the second case would lead to $x_1 \in c_{y',x'}([0,t'))$, which is a subset of $(X_-)_{t'} \subseteq (X_-)_{\delta/4}$ (recall $t' < \delta/4$), again contradicting the choice of $x_1$. Hence there can not be another intersection of the thickenings of $C$ with the upper side of $\partial B$ (other than in the same component as $x_0$), so we get the statement for a single ball $B = B_r^X(x_1)$.

We now turn to the case of an intersection $B := \bigcap_{i=1}^k B_r^X(x_i)$ of several balls. Again, $B \cap (C)_{\delta/4}$ is convex and the monotonicity statements for $d(\cdot, C)$ along the flow lines still hold. So we only have to show once again that there is no other intersection of the thickenings of $C$ with the upper side of $\partial B$ (other than in the component of the first contact point). As $\partial B$ consists of pieces of the boundaries of the balls $B_i$, such a second intersection point would again mean that for some $t' < \delta/4$, the thickening $(C)_{t'}$ would meet a boundary $\partial B_{i_0}$ tangentially at some point, for suitable $i_0 \in \{ 1, \ldots, k \}$. By the previous arguments for the case of a single ball (here: $B_{i_0}$), we get the desired contradiction.
\end{proof}

The following lemma is essential for our strategy. It basically says that if we take the cover given by balls and the stable covering sets as above (which, in general, is no longer good), then we can construct a good (stable) cover out of it by cutting the covering sets into smaller pieces; moreover, we can uniformly control the number of cuts needed and thus the number of sets in this refinement.

\begin{lem}\label{Lem:Verfeinerung_gute_Ueberdeckung}
There is a constant $\nu = \nu(n) \in \N$ only depending on $n$ with the following property. In the present situation, a set $B' = B_r^X(x) \setminus (X_-)_{\delta/4}$ can be replaced by at most $\nu$ open subsets $\{ U_{\omega} : \omega = 1, \ldots, \nu \} =: \mathcal{U}$, which form a good cover of $B'$. If the $B'$ are constructed from balls as above, then these sets $\mathcal{U}$ can be chosen compatible with each other, i.e. for a nonempty intersection of different $B'_1, \ldots, B'_k$, the intersections of the elements of $\mathcal{U}_1, \ldots, \mathcal{U}_k$ form a good cover of $\bigcap_{i=1}^k B'_i$. Here, we additionally assumed that the number of balls intersecting $B_r^X(x)$ nontrivially is bounded by a constant $\lambda = \lambda(n)$ only depending on $n$.
\end{lem}

\begin{proof}
We start the construction on the maximal intersections of the covering balls. More precisely: let $B'_1, \ldots, B'_k$ be sets as above, obtained from balls $B_r^X(x_i)$ via $B'_i = B_r^X(x_i) \setminus (X_-)_{\delta/4}$; if $\bigcap_{i=1}^k B_i \neq \emptyset$ (where $B_i := B_r^X(x_i)$), but $B_0 \cap \bigcap_{i=1}^k B_i = \emptyset$ for all other balls $B_0$ (from our finite family of balls), then the intersection $\bigcap_{i=1}^k B_i$ is maximal.

By Lemma \ref{Lem:Eigenschaften_Konstruktion_B'} we have
\[
\bigcap\limits_{i=1}^k B_i \setminus (X_-)_{\delta/4} = \bigcap\limits_{i=1}^k B_i \setminus \bigcup\limits_{j=1}^l (\{ d_{\gamma_j} < \epsilon \})_{\delta/4}
\]
for suitable $\gamma_1, \ldots, \gamma_l \in \Gamma \setminus \{ \id \}$ and $l \leq \kappa = \kappa(n)$. We will construct the refinement by induction on $l$.

For $l=1$, we use Lemma \ref{Lem:Durchschnitt_Kugeln_Ausschneidung_homoeomorph_Halbkugel} and choose a homeomorphism between $\bigcap_{i=1}^k B_i \cong D^n$ and an $n$-dimensional open cube $W$, where we think of $W$ as partitioned into $3^n$ subcubes of equal size; more precisely, the partition $(0,1) = (0,1/3] \cup [1/3, 2/3] \cup [2/3, 1)$ (which is not disjoint) will be used to get a partition of $(0,1)^n = W$ by taking products. Since by Lemma \ref{Lem:Durchschnitt_Kugeln_Ausschneidung_homoeomorph_Halbkugel}, $\bigcap_{i=1}^k B_i \cap (\{ d_{\gamma_1} < \epsilon \})_{\delta/4}$ sits inside $\bigcap_{i=1}^k B_i$ as half a ball, we can choose this homeomorphism in such a way that $\bigcap_{i=1}^k B_i \cap (\{ d_{\gamma_1} < \epsilon \})_{\delta/4}$ corresponds to a sufficiently small open neighborhood of one of the subcubes of the partition above; obviously, this subcube has to lie at the boundary of the cube. By passing to sufficiently small open neighborhoods of the remaining subcubes (after removing the open neighborhood corresponding to $(\{ d_{\gamma_1} < \epsilon \})_{\delta/4}$), we get a good cover of $\bigcap_{i=1}^k B_i \setminus (\{ d_{\gamma_1} < \epsilon \})_{\delta/4}$. Note that intersections of the covering sets correspond to (open neighborhoods of the) joint faces of the subcubes intersecting each other. The cube model we just constructed will also be called cube grid; Figure \ref{Bild_Wuerfelraster_Anfang} illustrates the situation.

{
\begin{figure}
\centering

\begin{tikzpicture}[scale=0.75]


\begin{scope}[xshift=-5cm]



\fill[color=gray!2, xshift=0cm, yshift=0cm] (0,0) .. controls (1.25,2.5) and (3.25, 3.25) .. (4,3.5) -- (4, 0) -- cycle; 



\fill[color=gray!8, xshift=0cm, yshift=0cm] (0,0) .. controls (1.25,2.5) and (3.25, 3.25) .. (4,3.5) -- (4, 6) -- (-3, 6) -- (-3, 0) -- cycle; 


\coordinate (x) at (-1,2.8); 
\coordinate (y) at (-0.2,3.5); 


\draw[ultra thin, xshift=0cm, yshift=0cm] (0,0) .. controls (1.25,2.5) and (3.25, 3.25) .. (4,3.5); 

\draw[ultra thin, fill=gray, fill opacity=0.1] (x) circle (1.5cm); 

\draw[ultra thin, fill=gray, fill opacity=0.1] (y) circle (1.5cm); 


\draw[very thin, dashed, xshift=0cm, yshift=0cm] (-2,0) .. controls (-0.5,3.5) and (1.5, 4.5) .. (4,5.5); 

\draw [decorate,decoration={brace,amplitude=5pt},xshift=0pt,yshift=0pt]
(1.95,4.55) -- (2.9,3.05) node [black,midway,xshift=12pt, yshift=5pt] 
{\scriptsize $\delta/4$};

{
\draw[thin, fill=gray, fill opacity=0.15] (-0.82,2.131) arc(245.6:290:1.5cm) -- (0.325,2.091) arc(-27.9:26.4:1.5cm) -- (0.338,3.46) .. controls (-0.2,2.94) .. (-0.82,2.131);
}


\node at (1.25,5.25) {$X_+$}; 
\node at (2.5,0.7) {$\{ d_{\gamma} < \epsilon \}$}; 


\node at (-0.8, 3.6) {\footnotesize $B_1 \cap B_2$};
\node at (-2, 2.5) {\footnotesize $B_1$};
\node at (0, 4.5) {\footnotesize $B_2$};

\draw[very thin] (0,2.4) -- (0,1.2);
\node at (-0.1,1) {\tiny $(B_1 \cap B_2)$};
\node at (0.1,0.7) {\tiny $\cap (\{ d_{\gamma} < \epsilon \})_{\delta/4}$};

\end{scope}


\begin{scope}[xshift=2cm, yshift=0.75cm]

\draw[fill=gray!3] (0,0) -- (4.5,0) -- (4.5,4.5) -- (0,4.5) -- cycle;

\draw[very thin] (1.5,0) -- (1.5,4.5);

\draw[very thin] (3,0) -- (3,4.5);

\draw[very thin] (0,1.5) -- (4.5,1.5);

\draw[very thin] (0,3) -- (4.5,3);

\draw[fill=gray!20] (1.5,0) -- (3,0) -- (3,1.5) -- (1.5,1.5) -- cycle;


\end{scope}

\end{tikzpicture}

\caption{
The highlighted area in the cube to the right corresponds to the set $(B_1 \cap B_2) \cap (\{ d_{\gamma} < \epsilon \} )_{\delta/4}$ in the left picture. Small open neighborhoods of the other subcubes yield a good cover of the rest of the cube, which corresponds to $(B_1 \cap B_2) \setminus ( \{ d_{\gamma} < \epsilon \} )_{\delta/4}$.}
\label{Bild_Wuerfelraster_Anfang}

\end{figure}
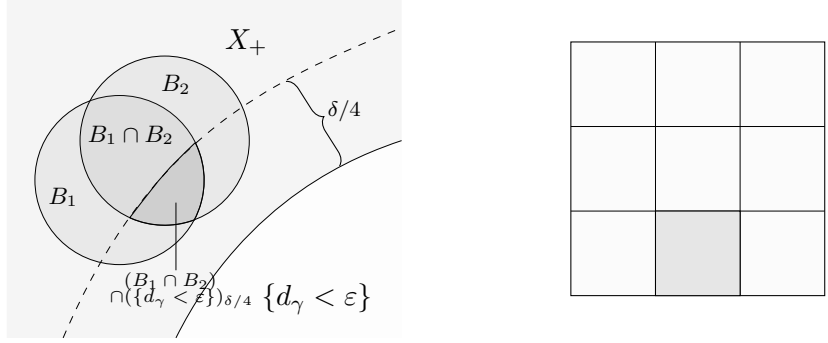
}

Consequently, for the step $l \rightarrow l+1$ we can assume that $\bigcap_{i=1}^k B_i$ is homeomorphic to a decomposition of the cube $W$ into finitely many subcubes (as above), where the subspace $\bigcap_{i=1}^k B_i \setminus \bigcup_{j=1}^l (\{ d_{\gamma_j} < \epsilon \})_{\delta/4}$ corresponds some of these subcubes (how many subcubes we need will be shown later on). Now, we have to find a cube grid for the situation that we further remove the set $(\{ d_{\gamma_{l+1}} < \epsilon \})_{\delta/4}$. To this end we map $(\{ d_{\gamma_{l+1}} < \epsilon \})_{\delta/4}$ into the existing cube grid under the present homeomorphism; next, we have to deform this homeomorphic image $C_{\delta/4}$ of $(\{ d_{\gamma_{l+1}} < \epsilon \})_{\delta/4}$ in $W$ in such a way that -- after a suitable refinement of the subcubes -- this is compatible with the existing cube grid, i.e. that $C_{\delta/4}$ can again be thought of as a union of (open neighborhoods of) subcubes.

Similar to the proof of Lemma \ref{Lem:Durchschnitt_Kugeln_Ausschneidung_homoeomorph_Halbkugel} let $x$ be the first contact point of the thickenings of $\{ d_{\gamma_{l+1}} < \epsilon \}$ with $\bigcap_{i=1}^k \overline{B_i}$; under the present homeomorphism, $x$ corresponds to a point $x'$ on the boundary $\partial W$ of the cube $W \subseteq \R^n$. Now $\partial W$ is made up of boundary pieces of the subcubes, hence $x'$ lies in the boundary of (at least) one subcube. Note that by convexity of $(\{ d_{\gamma_{l+1}} < \epsilon \})_{\delta/4}$, this set admits a (unique) fibering via geodesics from $x$ to its points; via the homeomorphism, we thus get a fibering of $W$ by curves from $x'$ to the points of $C_{\delta/4}$. By the definition of $x$ as the first contact point, we can think of this successive thickening of $\{ d_{\gamma_{l+1}} < \epsilon \}$ -- i.e. taking the sets $(\{ d_{\gamma_{l+1}} < \epsilon \})_t$ for increasing $0 < t \leq \delta/4$ -- as a flow away from $x$ along the fibering geodesics; similar observations hold for the situation in $W$ with a flow along the above curves away from $x'$. We will denote the image of $(\{ d_{\gamma_{l+1}} < \epsilon \})_t$ in $W$ by $C_t$ and call the above curves fibering curves. Note that by convexity of $(\{ d_{\gamma_i} < \epsilon \})_{\delta/4}$ ($i=1,\ldots, l+1$), the set $(\{ d_{\gamma_{l+1}} < \epsilon \})_{\delta/4} \cap \bigcup_{i=1}^l (\{ d_{\gamma_i} < \epsilon \})_{\delta/4}$ consists of at most $l$ connected components, so this also translates to the situation in $W$. Looking from $x'$, these $l$ components correspond to $l$ families of fibering curves, where two curves are in the same family if they meet the same connected component, see Figure \ref{Bild_Wuerfelraster_Hinzufuegen_und_Faserung_neuer_Menge}.

{
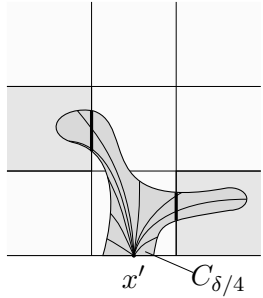
\begin{figure}
\centering

\begin{tikzpicture}[scale=0.75]


\fill[color=gray!20] (0,1.5) -- (1.5,1.5) -- (1.5,3) -- (0,3) -- cycle;

\fill[color=gray!20] (3,0) -- (4.5,0) -- (4.5,1.5) -- (3,1.5) -- cycle;

\fill[color=gray!3] (0,0) -- (3,0) -- (3,1.5) -- (4.5,1.5) -- (4.5,4.5) -- (0,4.5) -- (0,3) -- (1.5,3) -- (1.5,1.5) -- (0,1.5) -- cycle;


\draw[very thin] (1.5,0) -- (1.5,4.5);

\draw[very thin] (3,0) -- (3,4.5);

\draw[very thin] (0,1.5) -- (4.5,1.5);

\draw[very thin] (0,3) -- (4.5,3);

\draw (3,0) -- (3,1.5) -- (4.5,1.5);

\draw (0,1.5) -- (1.5,1.5) -- (1.5,3) -- (0,3);


\draw[fill=gray!30] (1.7,0) .. controls (1.9,0.7) and (2,1.8) .. (1.2,2) arc (250:70:0.5cm and 0.3cm) .. controls (2.3,2.3) and (1.5,0.75) .. (3.9,1.2) arc (100:-90:0.3cm and 0.2cm) .. controls (2.7,0.6) .. (2.6,0) -- cycle;


\draw[very thick] (1.5,1.85) -- (1.5,2.57);

\draw[very thick] (3,0.62) -- (3,1.12);


\draw[ultra thin] (2.25,0) .. controls (1.95,0.2) .. (1.77,0.3);

\draw[ultra thin] (2.25,0) .. controls (2,0.55) .. (1.835,0.75);

\draw[ultra thin] (2.25,0) .. controls (2,0.95) .. (1.78,1.4);

\draw[ultra thin] (2.25,0) .. controls (2.05,1.45) and (1.65,2) .. (1.1,2.03);

\draw[ultra thin] (2.25,0) .. controls (2.2,1.55) and (1.5,2.25) .. (1.25,2.57);

\draw[ultra thin] (2.25,0) .. controls (2.35,0.8) .. (2.35,1.32);

\draw[ultra thin] (2.25,0) .. controls (2.35,0.6) and (2.8,0.95) .. (3.1,1.12);

\draw[ultra thin] (2.25,0) .. controls (2.4,0.6) and (3.1,1) .. (4.25,1);

\draw[ultra thin] (2.25,0) .. controls (2.4,0.2) .. (2.67,0.33);

\draw (0,0) -- (4.5,0);


\filldraw (2.25,0) circle (0.75pt) node[below] {$x'$};

\draw[very thin] (2.45,0.07) -- (3.3,-0.3);
\node at (3.75,-0.45) {$C_{\delta/4}$};


\end{tikzpicture}

\caption{
Inserting $C_{\delta/4}$ into an existing cube grid. The "fibering" of $C_{\delta/4}$ is shown via the small curves starting in $x'$. Curves which intersect one of the bold lines -- i.e. intersect the cubes to be removed -- belong to the same family.}
\label{Bild_Wuerfelraster_Hinzufuegen_und_Faserung_neuer_Menge}

\end{figure}
}

If such a family of fibering curves does not end inside the part of the cube grid to be removed, then we can push these curves back to the last position where they left the part to be removed; and if they don't meet this part at all, we can push them back to $x'$. This can be realized by a suitable homeomorphism of the cube grid. Hence up to homeomorphism, we can assume that all fibering curves which intersect the part to be removed also end in the interior of it; and if they don't meet the part to be removed, we can assume they remain inside a sufficiently small neighborhood of $x'$.

Recall that the fibering curves of $C_{\delta/4}$ are the homeomorphic image of geodesics in $X$; we will use this to straighten the parts of $C_{\delta/4}$ which meet the subcubes to be removed. Take a family of fibering curves as above, ending inside a common component of the part to be removed. On their way from $x'$ to that component they might enter and leave the part to be removed several times. We will now think of these parts of curves between those cubes to be removed as separate pieces, i.e. such a piece will end in the part to be removed and either start in $x'$ or in the part to be removed. Now observe that every such piece can be deformed homeomorphically such that a subcube of the grid which should \textit{not} be removed will not be entered again after leaving it for the first time; this can be seen as a straightening of the fibering curves.

Using the two steps above\footnote{I.e. pushing back the ends of the fibering curves of $C_{\delta/4}$ not ending in the part to be removed to $x'$ or the position where they last left a subcube to be removed; and straightening the fibering curves between the positions where they leave and enter the part to be removed.} we can ensure that for a subcube which should not be removed, the intersection of its boundary with $C_{\delta/4}$ will have at most $2l$ components: for each of the $l$ directions from $x'$ to the intersections of $C_{\delta/4}$ not more than one intersection for entering and leaving, each. In order to describe the homeomorphism type of $C_{\delta/4}$ inside such a subcube, for each dimension we will need at most $2l$ new pieces of the respective interval to account for the $2l$ directions, at which $C_{\delta/4}$ might leave the subcube, and additionally at most $2l+1$ pieces to account for space in between those directions. Hence every interval (whose product yields the cube $W$) has to be split into at most $4l+1$ parts; thus in dimension $n$, we will split each of the subcubes into not more than $(4l+1)^n \leq (4\kappa + 1)^n$ new subcubes. As this procedure might be necessary for every of the subcubes, we conclude that in the step $l \rightarrow l+1$, the number of subcubes increases at most by factor $(4l+1)^n \leq (4\kappa + 1)^n$. Recall that in step $l=1$, we had $3^n$ subcubes; since there will be at most $\kappa$ steps, we have to repeat the procedure $\leq \kappa - 1$ times. Hence the number of subcubes after all possible steps will be bounded by
\[
3^n \cdot \prod\limits_{l = 2}^{\kappa} (4l + 1)^n \leq 3^n \cdot \prod\limits_{l = 2}^{\kappa} (4\kappa + 1)^n \leq 3^n \cdot (4\kappa + 1)^{n\cdot\kappa} =: \nu_0 = \nu_0(n) \in \N.
\]
As mentioned above, we get a good cover of the remaining cube by taking sufficiently small open neighborhoods of the subcubes not to be removed; by using the final homeomorphism, this yields a good cover for the set in $X$.

We will now extend the cube grid constructed for a maximal intersection $\bigcap_{i=1}^k B_i$ as above successively to the entire ball $B_1 = B_r^X(x)$. Note that $\bigcap_{i=1}^k B_i$ is a convex subset of the next larger intersection $\bigcap_{i=1}^{k-1} B_i$; if we apply the above procedure of cutting into subcubes to $\bigcap_{i=1}^{k-1} B_i$, then we will keep the cube grid which we already obtained on $\bigcap_{i=1}^k B_i \subseteq \bigcap_{i=1}^{k-1} B_i$ fixed and will refine (i.e. cut into subcubes) only the other parts. We proceed similarly for the other sets obtained by leaving out one or several of the $B_i$ when taking the intersection. In this fashion, we will get a good cover of $B_1$ which will be compatible with the good covers of the other balls $B_i$, as they coincide on their common intersection (see Figure \ref{Bild_Wuerfelraster_Erweitern}).

{
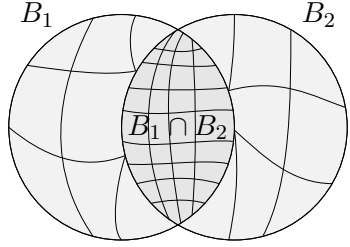
\begin{figure}
\centering

\begin{tikzpicture}[scale=0.75]


\draw[very thin, fill=gray, fill opacity=0.1] (-1,0) circle (2cm);

\draw[very thin, fill=gray, fill opacity=0.1] (1,0) circle (2cm);


\draw[very thin] (-1,0) circle (2cm);

\draw[very thin] (1,0) circle (2cm);


\draw[very thin] (-1.5,1.94) .. controls (-2,1) and (-2.5,-0.5) .. (-1.5,-1.94);

\draw[very thin] (-0.75,1.98) .. controls (-0.9,1.7) and (-0.85,1.25) .. (-0.75,1);

\draw[very thin] (-0.75,1) .. controls (-1.25,0.9) and (-2,1) .. (-2.68,1.1);

\draw[very thin] (-0.93,-0.5) .. controls (-1.1,-1) and (-1.2,-1.8) .. (-1,-2);

\draw[very thin] (-0.93,-0.5) .. controls (-1.7,-0.8) and (-2.4,-0.5) .. (-3,0);

\draw[very thin] (2,1.72) .. controls (2,0.5) and (2,-1) .. (1.5,-1.94);

\draw[very thin] (1,2) .. controls (1,1.75) and (1,1.25) .. (0.9,0.65);

\draw[very thin] (0.9,0.65) .. controls (1.8,0.9) and (2.4,0.55) .. (2.97,0.35);

\draw[very thin] (1,-0.05) .. controls (1.1,-0.9) and (1,-1.5) .. (0.9,-2);

\draw[very thin] (1,-0.05) .. controls (1.5,-0.5) and (2,-1) .. (2.74,-1);

\draw[ultra thin] (-0.43,1.4) .. controls (-0.1,1.35) and (0.1,1.3) .. (0.43,1.4);

\draw[ultra thin] (-0.63,1.15) .. controls (-0.2,1.1) and (0.4,1.05) .. (0.68,1.1);

\draw[ultra thin] (-0.87,0.7) .. controls (-0.4,0.65) and (0.4,0.65) .. (0.87,0.7);

\draw[ultra thin] (-0.98,0.3) .. controls (-0.25,0.15) and (0.4,0.35) .. (0.98,0.3);

\draw[ultra thin] (-0.98,-0.25) .. controls (-0.5,-0.4) and (0.4,-0.2) .. (0.98,-0.25);

\draw[ultra thin] (-0.9,-0.65) .. controls (-0.45,-0.7) and (0.5,-0.55) .. (0.9,-0.65);

\draw[ultra thin] (-0.73,-1) .. controls (-0.3,-0.9) and (0.2,-0.85) .. (0.73,-1);

\draw[ultra thin] (-0.48,-1.35) .. controls (-0.25,-1.3) and (0.15,-1.25) .. (0.48,-1.35);


\draw[ultra thin] (-0.15,-1.65) .. controls (-0.8,-0.2) and (-0.5,1.2) .. (-0.05,1.7);

\draw[ultra thin] (0.05,-1.7) .. controls (-0.2,-0.3) and (-0.3,1.1) .. (0.1,1.68);

\draw[ultra thin] (0.25,-1.55) .. controls (0.35,-0.4) and (0.3,1) .. (0.25,1.55);


\node at (-2.5, 2) {$B_1$};
\node at (0, 0) {$B_1 \cap B_2$};
\node at (2.5, 2) {$B_2$};

\end{tikzpicture}

\caption{
Extending the cube grid of $B_1 \cap B_2$ to a cube grid on $B_1$ and $B_2$.}
\label{Bild_Wuerfelraster_Erweitern}

\end{figure}
}

It remains to determine how many new sets will be needed to cover the initial ball $B_1$. Using that by assumption, there are at most $\lambda = \lambda(n)$ other balls intersecting $B_1$, we see that there are at most
\[
\sum\limits_{i=1}^{\lambda} 2^i \leq \sum\limits_{i=1}^{\lambda} 2^{\lambda} = \lambda \cdot 2^{\lambda}
\]
different intersections of different size for these balls. Thus we will need a refinement into at most $\lambda \cdot 2^{\lambda} \cdot \nu_0 =: \nu$ subsets. As $\lambda$ and $\nu_0$ depend only on $n$, the same is true for $\nu$.
\end{proof}

Lemma \ref{Lem:Verfeinerung_gute_Ueberdeckung} above is based on the assumption that the number of intersections between the balls can be controlled by $\lambda = \lambda(n)$. Later on, we will see that such a bound follows easily from the fact that the centers of the balls have a uniform lower bound on their distance from one another.

\subsection{Nerve construction}

We will now explain how to cover the (shrunken) thick part $M'_+$ by a good open cover which is stable under the flow, using the constructions from the previous sections. Let $\mathfrak{M}$ be a maximal $(r/2)$-discrete subset of $M'_+$. As before, for a ball $B_r^M(p)$ (where $p \in \mathfrak{M}$) with $B_r^M(p) \setminus M'_+ \neq \emptyset$, denote the corresponding stable covering set -- obtained by cutting off the part outside $M'_+$ and extending the remainder along the flow lines -- by $B'$. Let $N_+$ be the union of all these sets, so $N_+$ consists of the balls $B_r^M(p)$ lying completely inside $M'_+$ and the changed covering sets $B'$, whose initial balls had parts outside of $M'_+$. Moreover, let $N_0 := N_+ \cap \overline{(M_-)_{\delta/4}}$. By construction,
\[
M'_+ \subseteq N_+ \subseteq M_+ \qquad\text{and}\qquad \partial M'_+ \subseteq N_0.
\]
The stability of the covering is crucial for the following Lemma.

\begin{lem}\label{Lem:Sauer3.14Analogon}
The inclusions $j: M'_+ \hookrightarrow N_+$ and $\partial j: \partial M'_+ \hookrightarrow N_0$ are homotopy equivalences.
\end{lem}

\begin{proof}
Again, the proof is analogous to \cite{Sauer} Lemma 3.14: for the map
\[
r: N_+ \hookrightarrow M_+ \stackrel{f(\cdot,1)}{\rightarrow} M'_+
\]
we have $r \circ j = \id_{M'_+}$ and using the stability of the covering sets under the flow $f$, we also see that $j \circ r$ is homotopic to $\id_{N_+}$.

The argument can be repeated for $\partial j$, as $f(\cdot,t)$ can be restricted to $\overline{(M_-)_{\delta/4}}$ (for all $t\in [0,1]$).
\end{proof}

By Lemma \ref{Lem:Verfeinerung_gute_Ueberdeckung}, we possibly have to refine the remaining parts of shape $B_r^X(x) \setminus (X_-)_{\delta/4}$ of the cut off balls into at most $\nu$ subsets to obtain a good cover. We do this now and similarly for their corresponding balls $B_r^M(p)$ in $M$. Again, just as for the initial balls, we can extend the resulting subsets which meet the boundary of $M'_+$ along the flow lines to obtain a good cover $\mathcal{U}$ of all of $N_+$.

For $U \in \mathcal{U}$ with $U \cap \partial M'_+ \neq \emptyset$ define $U' = U \cap \overline{(M_-)_{\delta/4}}$ and, finally, $\mathcal{V}$ as the set containing all these $U'$. By construction of the refinement, $\mathcal{V}$ is a good cover of $N_0$. Let $N(\mathcal{U})$ and $N(\mathcal{V})$ denote the corresponding nerve complexes (see e.g. \cite{Sauer} section 2.2 for a recapitulation). Since every $U' \in \mathcal{V}$ corresponds to precisely one $U \in \mathcal{U}$ and the same holds true for the intersections of different $U'$ with respect to intersections of the corresponding $U$, we can consider $N(\mathcal{V})$ as a subcomplex of $N(\mathcal{U})$. We obtain the following crucial lemma.

\begin{lem}\label{Lem:Homotopieaequivalenz_zu_Simplizialkomplex}
The pair $(M_+, \partial M_+)$ is (as a pair) homotopy equivalent to the simplicial pair $(N(\mathcal{U}), N(\mathcal{V}))$.
\end{lem}

\begin{proof}
Using Lemma \ref{Lem:Sauer3.13Analogon} and \ref{Lem:Sauer3.14Analogon} as substitutes for the analogous statements of \cite{Sauer}, the proof given in \cite{Sauer} (between Lemma 3.14 and Lemma 3.15) can be repeated in our situation. Note that this proof only uses homotopy theoretic arguments, so our different curvature conditions won't matter.
\end{proof}

In order to obtain our main result, it remains to show that the complexity of $N(\mathcal{U})$ can be bounded. To this end, we introduce the following notation from \cite{Sauer}: a simplicial complex $S$ is a \textbf{$(D,C)$-simplicial complex} (where $D, C > 0$), if the number of vertices of $S$ is bounded by $C$ and the degree at each vertex is at most $D$. Similarly, a \textbf{$(D,C)$-simplicial pair} is a simplicial pair $(S,S')$, where $S$ is a $(D,C)$-simplicial complex (thus $S'$ satisfies the same bounds).

Define constants $C = C(n), D = D(n) > 0$ via
\begin{align*}
C &:= \frac{\nu}{\Vol_{\R^n}(B_{r/4}^{\R^n})},\\
D &:= \nu \cdot N(n, r/2, 2r),
\end{align*}
where $N(n, r/2, 2r)$ is as in \cite{Sauer} Lemma 3.2 and $\nu$ is taken from Lemma \ref{Lem:Verfeinerung_gute_Ueberdeckung}; note that these values depend only on $n$, as $r = \epsilon(n)/32$ and $\nu = \nu(n)$ only depended on $n$. The complexity of $N(\mathcal{U})$ is bounded as follows:

\begin{lem}\label{Lem:Sauer3.15Analogon}
$N(\mathcal{U})$ is a $(D, C \cdot \Vol(M))$-simplicial complex.
\end{lem}

\begin{proof}
We will first derive a bound on the cardinality of a general $(r/2)$-discrete subset $A$ of $M_+$; this also bounds the cardinality of $\mathfrak{M}$. If $x \in X_+$ is a preimage of some $p \in A$, then $d_{\Gamma}(x) \geq \epsilon = \epsilon(n)/4$. By choice of $r = \epsilon(n)/32 = \epsilon/8$, we see that for all $0 < \rho \leq 2r$ the ball $B_{\rho}^M(p)$ is isometric to $B_{\rho}^X(x)$; in particular, their volumes coincide. Using the curvature assumptions, we conclude that $\Vol_X(B_{\rho}^X(x)) \geq \Vol_{\R^n}(B_{\rho}^{\R^n})$. Moreover, for all $\rho \leq r/4$, the $\rho$-balls around the points in $A$ are disjoint. Hence $\{ B_{r/4}^M(p) : p \in A \}$ is a family of disjoint balls of volume at least $\Vol_{\R^n}(B_{r/4}^{\R^n}) =: V(n)$ each. Thus
\[
|\mathfrak{M}| \leq |A| \leq \frac{\Vol(M)}{V(n)}.
\]
Constructing $\mathcal{U}$, we cut a ball around some $p \in \mathfrak{M}$ in at most $\nu$ pieces, hence
\[
|\mathcal{U}| \leq \nu \cdot |\mathfrak{M}| \leq \nu \cdot \frac{\Vol(M)}{V(n)} = C \cdot \Vol(M).
\]
We now turn to the degree bound $D$. Recall that in the construction of $\mathcal{U}$, we started with balls $B_r^M(p)$ for $p \in \mathfrak{M}$. Assume two such radius-$r$-balls around $p, p' \in \mathfrak{M}$ intersect, then we find lifts $x, x' \in X$ of $p, p'$ such that $d_X(x,x') < 2r$. The lift of $\mathfrak{M}$ in $X$ is also $(r/2$)-discrete, hence the set of lifts whose balls intersect $B_r^X(x)$ is a $(r/2)$-discrete subset of $B_{2r}^X(x)$. By \cite{Sauer} Lemma 3.2, that cardinality is bounded by $N(n, r/2, 2r)$; in other words, at most $N(n, r/2, 2r)$ many $r$-balls around points of $\mathfrak{M}$ can have a nonempty intersection. As $N(n, r/2, 2r)$ depends only on $n$, this defines the constant $\lambda = \lambda(n)$ needed for Lemma \ref{Lem:Verfeinerung_gute_Ueberdeckung}. Again using that a ball was cut into at most $\nu$ pieces during the construction of $\mathcal{U}$, we get that a set in $\mathcal{U}$ intersects at most
\[
\nu \cdot N(n, r/2, 2r) = D
\]
other covering sets.
\end{proof}

Summarizing all the above statements, we have proven our main result:

\begin{satz}\label{Satz:Hauptresultat_Sichtbarkeit}
$(M_+, \partial M_+)$ is (as a pair) homotopy equivalent to a $(D, C\cdot\Vol(M))$-simplicial pair, where $C = C(n)$ and $D = D(n)$ are constants depending only on the dimension $n$. {\hfill$\square$}
\end{satz}

\subsection{Applications}

Bounds on the homology are straightforward consequences of the main result Theorem \ref{Satz:Hauptresultat_Sichtbarkeit}.

\begin{satz}\label{Satz:Freier_Anteil_beschraenkt_durch_Volumen}
There is a constant $E = E(n) > 0$ depending only on $n$, such that
\[
b_k(M;\K) \leq E \cdot \Vol(M)
\]
for all $k\in\N_0$, where $b_k(M;\K) = \dim_{\K} H_k(M;\K)$ is the $k$-th Betti number of $M$ with coefficients in the field $\K$ (of arbitrary characteristic).
\end{satz}

\begin{proof}
The proof is a standard argument utilizing the Mayer-Vietoris sequence. Note that it is sufficient to construct constants $E(k,n)$ also depending on the respective degree $k=0,\ldots,n$, since $E = \max_k E(k,n)$ will then satisfy the initial statement.

As $M$ is connected, $b_0(M;\K) = 1$. Hence $E(0,n) := 1/V(n)$ yields the statement for $k=0$, where $V(n)$ is a uniform lower bound on the volume of all complete Riemannian $n$-manifolds with $-1\leq K < 0$ (\cite{BGS} Corollary 8.4).

By the thick-thin decomposition Theorem \ref{Satz:DickDuennZerlegungSichtbarkeit}, $M$ is homotopy equivalent to its compact part $M_C$, which in turn is obtained by gluing the tubes to the thick part $M_+$ along their common boundary $\partial_{\mathcal{T}} M_+$. Let $\mathcal{T}$ denote the set of all tubes of $M$. By Mayer-Vietoris we thus have an exact sequence
\begin{center}
\begin{tikzcd}[column sep = small]
\ldots \ar[r]		&		H_k( \bigcup\limits_{T\in \mathcal{T}} T; \K) \oplus H_k(M_+; \K) \ar[r, "\alpha"]		&		H_k(M_C; \K) \ar[r, "\beta"]		&		H_{k-1}( \partial_{\mathcal{T}} M_+; \K) \ar[r]		&		\ldots
\end{tikzcd}
\end{center}
and hence
\begin{equation}\label{Gleichung:Dimensionen_Mayer_Vietoris}
\dim_{\K} H_k(M_C; \K) = \dim_{\K} \im (\beta) + \dim_{\K} \im (\alpha).
\end{equation}
Note that for a general $(A,B)$-simplicial complex, the number of $k$-simplices is bounded by $A^k \cdot B$. By our main result Theorem \ref{Satz:Hauptresultat_Sichtbarkeit} we thus get
\[
\dim_{\K} H_{k-1}( \partial_{\mathcal{T}} M_+; \K) \leq D^{k-1} \cdot C \cdot \Vol(M).
\]
Similarly,
\[
\dim_{\K} H_k(M_+; \K) \leq D^k \cdot C \cdot \Vol(M).
\]
As there are at most $C \cdot \Vol(M)$ many tubes\footnote{The $C$ used here, i.e. the one from the thick-thin decomposition Theorem \ref{Satz:DickDuennZerlegungSichtbarkeit}, is a different one than the $C$ from the main result Theorem \ref{Satz:Hauptresultat_Sichtbarkeit}. Still, we just use $C$ for both of them, as we could simply take the maximum.}, the tubes are disjoint and all homotopy equivalent to $\S^1$, we have
\[
\dim_{\K} H_k( \bigcup\limits_{T\in \mathcal{T}} T; \K) \leq C \cdot \Vol(M).
\]
Hence using (\ref{Gleichung:Dimensionen_Mayer_Vietoris}), we deduce
\[
\dim_{\K} H_k(M_C; \K) \leq (D^{k-1} + D^k + 1) \cdot C \cdot \Vol(M).
\]
As $M_C \simeq M$, the values $E(k,n) := (D^{k-1} + D^k + 1) \cdot C$ satisfy the statement.
\end{proof}

A similar result holds for the torsion part of the homology.

\begin{satz}\label{Satz:Torsionsanteil_beschraenkt_durch_Volumen}
There is a constant $F = F(n) > 0$ depending only on $n$, such that
\[
\log | \tors H_k(M;\Z) | \leq F \cdot \Vol(M)
\]
for all $k\in\N_0$, where for $n=3$ the case $k=1$ has to be excluded.
\end{satz}

\begin{proof}
Using our analogous statements, the proof is literally the same as the one for \cite{Sauer} Theorem 1.2. Note that in \cite{Sauer}, the case of $n=3$ is universally excluded, although the proof given there also holds for all degrees $k \neq 1$ in dimension $n=3$.
\end{proof}

\cite{Sauer} also gives an example of why torsion in the first homology can in general not be bounded in dimension $3$.

Finally, we state an analogue to \cite{Sauer} Theorem 1.5; let $\mathfrak{Htp}_n(V)$ denote the number of homotopy classes of $n$-dimensional complete Riemannian visibility manifolds of volume at most $V < \infty$ and with sectional curvature $-1 \leq K < 0$.

\begin{satz}\label{Satz:Homotopie-und_Homoeomorphietypen_zaehlen}
For $n \geq 4$, there exist constants $\alpha = \alpha(n), \beta = \beta(n) > 0$ only depending on $n$, such that
\[
\alpha \cdot V \cdot \log V \leq \log \mathfrak{Htp}_n(V) \leq \beta \cdot V \cdot \log V
\]
for sufficiently large $V > 0$.
\end{satz}

\begin{proof}
Again, the proof is the same as for \cite{Sauer} Theorem 1.5.
\end{proof}

Note that we are not able to give a statement analogous to \cite{Sauer} Corollary 1.6 -- i.e. counting the number of homeomorphism types --, as the tool used for that proof only holds in strictly negative curvature.

\section{Examples of negatively curved visibility manifolds with curvature not bounded away from zero}\label{Kapitel:Beispiele_Sichtbarkeit}

In \cite{JiWu}, Ji and Wu constructed a complete, finite volume visibility surface with sectional curvature $-1 \leq K < 0$ which does not satisfy the pinching condition\footnote{In fact they showed that the surface is not Gromov hyperbolic; recall that a manifold with pinched metric would already be Gromov hyperbolic.} $-1 \leq K \leq a < 0$ for any $a < 0$. We will extend this example to arbitrary dimensions $n > 2$.

Note that this does not provide an answer to the question whether every complete, finite volume visibility manifold with sectional curvature $-1 \leq K < 0$ can also be equipped with a complete, finite volume pinched metric. In fact, the surface in \cite{JiWu} was constructed out of a hyperbolic surface by flattening the hyperbolic metric along the (single) cusp, so it does \textit{not} serve as an example to separate these two classes of manifolds up to homeomorphism or diffeomorphism.

Similarly, to obtain our examples in dimension $n > 2$, we will start with a suitable non-compact hyperbolic $n$-manifold of finite volume and flatten the metric along a single cusp. The major difficulty compared to the $2$-dimensional case obviously lies in the more delicate computation of all the sectional curvatures. We remark that our resulting manifolds will satisfy the curvature condition $-11 \leq K < 0$, so technically, they have to be rescaled to fulfill the initially mentioned assumption $-1 \leq K < 0$.

Let $M$ be a hyperbolic $n$-manifold of finite volume which is not compact, i.e. has at least one cusp $S$; we can assume that $S$ has an $(n-1)$-dimensional torus $\T^{n-1}$ as its cusp cross section (see e.g. \cite{LongReid}). If $g$ denotes the Riemannian metric of $M$, then on $S \cong \T^{n-1} \times [0,\infty)$ it takes the form of a warped product metric
\[
\exp(-2t)\, ds^2 + dt^2,
\]
where $ds^2$ is the flat metric on the $\T^{n-1}$-factor.

Just as in \cite{JiWu}, let $h: [0,\infty) \rightarrow \R$ be a smooth function with the following properties:
\begin{enumerate}
\item $h$ is positive and monotonically decreasing,
\item $h''/h$ is positive and monotonically decreasing,
\item $h(t) = \exp(-t)$ for $0 \leq t \leq 1$,
\item $h(t) = 1/t^2$ for $t\geq 3$.
\end{enumerate}
The idea is to replace the warping function $\exp(-t)$ of the hyperbolic metric by $h(t)$. This can be seen as flattening the cusp, since $h(t) = 1/t^2$ (for large $t$) will decrease much more slowly than $\exp(-t)$. More precisely, on $S \cong \T^{n-1} \times [0,\infty)$, we pass to the warped product metric
\[
h^2(t)\, ds^2 + dt^2
\]
and hereby get a new metric $g'$ on all of $M$ (where $g = g'$ on $M \setminus S$). Note that $g'$ is smooth and complete for the same reasons as in \cite{JiWu}.

We will now show that $M$ still has finite volume with respect to $g'$. Letting $M' := M \setminus \T^{n-1} \times (3,\infty)$, we get
\[
M' = M'' \cup \bigcup\limits_{i=1}^k S_i,
\]
where the $S_i$ are the possible other (open) cusps of $M$ (different from $S$) and $M''$ denotes the compact part of $M$, which is obtained from $M$ by removing the $S_i$ (for $i=1,\ldots,k$) and the part $\T^{n-1} \times (3,\infty) \subseteq \T^{n-1} \times [0,\infty) \cong S$ of $S$. Since the metric on the $S_i$ did not change, we still have $\Vol_{g'}(S_i) = \Vol_g(S_i) < \infty$; moreover, $\Vol_{g'}(M'') < \infty$ by compactness. Hence $\Vol_{g'}(M') < \infty$ and we deduce
\begin{align*}
\Vol_{g'}(M) &= \Vol_{g'}(M') + \Vol_{g'}(\T^{n-1} \times (3,\infty)) \\
&= \Vol_{g'}(M') + \int_{\T^{n-1}} \int_{3}^{\infty} h(t) \,dt \,ds \\
&= \Vol_{g'}(M') + \int_{\T^{n-1}} \int_{3}^{\infty} \frac{1}{t^2} \,dt \,ds \\
&= \Vol_{g'}(M') + \int_{\T^{n-1}} \frac{1}{3} \,ds \\
&< \infty,
\end{align*}
because the last integral is also finite by compactness of $\T^{n-1}$.

The curvature computations are contained in the following lemma.

\begin{lem}\label{Lem:Schnittkruemmung_neue_Metrik}
Let $K_M$ denote the sectional curvature on $M$ with respect to the new metric $g'$. Then:
\begin{itemize}
\item Outside of $S$ as well as for $t\leq 1$ on $S$ we have $K_M = -1$.
\item For $1 < t < 3$ we have $-11 \leq K_M \leq -0.04$ (independent of $t$).
\item For $3 \leq t$ we have $- \frac{6}{t^2} \leq K_M \leq - \frac{4}{t^2}$ (depending on $t$).
\end{itemize}
\end{lem}

\begin{proof}
As the metric hasn't changed outside of $S$, all sectional curvatures will still be $-1$ there. To compute the curvature on $S$, we will use the following formula.

\begin{adjustwidth}{25pt}{0pt}
\textbf{Fact.} (\cite{BO'N} S. 26) Let $(B, \langle \cdot, \cdot \rangle_B)$ and $(F, \langle \cdot, \cdot \rangle_F)$ be Riemannian manifolds and $f$ a positive, smooth function on $B$. Moreover, let $M = B \times_f F$ be the warped product with respect to $f$, i.e. in a point $m = (b,p) \in B \times F$ the metric $\langle \cdot, \cdot \rangle_M$ takes the form
\[
\langle X + V, Y + W \rangle_M = \langle X, Y \rangle_B + f^2(b) \cdot \langle V, W \rangle_F
\]
for $X + V, Y + W \in T_b B \oplus T_p F \cong T_m M$. If $\{ X+V, Y+W \}$ is an orthonormal basis of a tangent plane $\Pi$ in $T_m M$, then the sectional curvature $K_M(\Pi)$ of $M$ in $m = (b,p)$ along $\Pi$ satisfies the equation

\begin{equation}\label{Gleichung:Formel_Schnittkruemmung}
\begin{aligned}
K_M(\Pi) &= K_B(X,Y) \cdot \| X \wedge Y \|^2_B \\
 &\quad - f(b) \cdot \Big[ \langle W, W \rangle_F \cdot \nabla^2 f(X,X) - 2 \langle V, W \rangle_F \cdot \nabla^2 f(X,Y) \\
 &\hspace{6.1cm}+ \langle V, V \rangle_F \cdot \nabla^2 f(Y,Y) \Big] \\
 &\quad + f^2(b) \cdot \Big[ K_F(V,W) - \|(\grad f) (b)\|^2_B \Big] \cdot \| V \wedge W \|^2_F.
\end{aligned}
\end{equation}

Here $K_B$ and $K_F$ denote the sectional curvatures of $(B, \langle \cdot, \cdot \rangle_B)$ and $(F, \langle \cdot, \cdot \rangle_F)$ respectively, $\nabla^2 f$ the Hessian of $f$ and $(\grad f)$ the gradient of $f$. Furthermore, the terms $\| X \wedge Y \|^2_B$ and $\| V \wedge W \|^2_F$ are to be understood as follows: if $U$ is a vector space with scalar product $\langle \cdot, \cdot \rangle$, then
\[
\langle U_1 \wedge U_2, U_3 \wedge U_4 \rangle := \det \left( \begin{matrix} \langle U_1, U_3 \rangle & \langle U_1, U_4 \rangle \\ \langle U_2, U_3 \rangle & \langle U_2, U_4 \rangle \end{matrix} \right)
\]
is a canonical scalar product on the exterior power $\Lambda^2 (U)$.
\end{adjustwidth}

In our situation, $B = [0,\infty)$, $F = \T^{n-1}$, $f=h$, $\langle \cdot, \cdot \rangle_M = g'$ and the above formula (\ref{Gleichung:Formel_Schnittkruemmung}) can be simplified as follows. The first line on the right hand side vanishes because $[0,\infty)$ is $1$-dimensional and thus $K_B(X,Y) = 0$. Since the torus is flat, we also have $K_F(V,W) = 0$ in the last line. Note that we can assume $X = 0$ without restriction\footnote{If $V$ is an $n$-dimensional vector space and $U \subseteq V$ a $2$-dimensional subspace with basis $\{ u, u' \}$, we can construct an orthonormal basis $\{ \widetilde{u}, \widetilde{u}' \}$ of $U$ with $\widetilde{u}_1 = 0$ by letting $\widetilde{u} := u - (u_1/u'_1) \cdot u'$ and extending $\widetilde{u}$ to an orthonormal basis via the Gram-Schmidt process (and if $u_1$ or $u'_1$ is $0$, we can directly use Gram-Schmidt).}. Hence
\[
K_M(\Pi) = - f(b) \cdot \langle V, V \rangle_F \cdot \nabla^2 f(Y,Y) - f^2(b) \cdot \| (\grad f) (b) \|^2_B \cdot \| V \wedge W \|^2_F.
\]
As $f(b)$ is $h(t)$ in our situation (i.e. we will also write $(t,p)$ for $m = (b,p)$), we get
\[
\nabla^2 f = h'' \qquad\text{and}\qquad (\grad f) = h'.
\]
This yields
\[
K_M(\Pi) = - h(t) \cdot \langle V, V \rangle_{\T^{n-1}} \cdot h''(t) \cdot Y^2 - h^2(t) \cdot |h'(t)|^2 \cdot \| V \wedge W \|^2_{\T^{n-1}},
\]
where $Y \in \R$ (see above, $B = [0,\infty)$ is $1$-dimensional). Since $X + V = V$ was assumed to be normalized with respect to $\langle \cdot, \cdot \rangle_M$ -- i.e. $\langle V, V \rangle_M = 1$ --, utilizing the formula
\[
\langle V, W \rangle_M = h^2(t) \cdot \langle V, W \rangle_{\T^{n-1}} \qquad\text{for general } V, W \in T_p \T^{n-1} \subseteq T_m M
\]
(see the definition of the warped product) we get
\[
\langle V, V \rangle_{\T^{n-1}} = \frac{1}{h^2(t)} \cdot \langle V, V \rangle_M = \frac{1}{h^2(t)}.
\]
By the orthogonality of $X + V = V$ and $Y + W$ with respect to $\langle \cdot, \cdot \rangle_M$ we can also deduce
\begin{align*}
0 &= \langle X + V, Y + W \rangle_M \\
&= \langle X, Y \rangle_B + h^2(t) \cdot \langle V, W \rangle_{\T^{n-1}} \\
&= h^2(t) \cdot \langle V, W \rangle_{\T^{n-1}},
\end{align*}
i.e. $\langle V, W \rangle_{\T^{n-1}} = 0$ because $h > 0$. Since $Y + W$ was normalized with respect to $\langle \cdot, \cdot \rangle_M$, also
\begin{align*}
1 &= \langle Y + W, Y + W \rangle_M \\
&= \langle Y, Y \rangle_B + h^2(t) \cdot \langle W, W \rangle_{\T^{n-1}} \\
&= Y^2 + h^2(t) \cdot \langle W, W \rangle_{\T^{n-1}},
\end{align*}
holds, i.e.
\[
\langle W, W \rangle_{\T^{n-1}} = \frac{1 - Y^2}{h^2(t)}.
\]
Inserting all these terms into the definition of $\| V \wedge W \|^2_{\T^{n-1}}$, we get
\begin{align*}
\| V \wedge W \|^2_{\T^{n-1}} &= \langle V \wedge W, V \wedge W \rangle_{\T^{n-1}} \\
&= \det \left( \begin{matrix} \langle V, V \rangle_{\T^{n-1}} & \langle V, W \rangle_{\T^{n-1}} \\ \langle W, V \rangle_{\T^{n-1}} & \langle W, W \rangle_{\T^{n-1}} \end{matrix} \right) \\
&= \det \left( \begin{matrix} 1/h^2(t) & 0 \\ 0 & (1-Y^2)/h^2(t) \end{matrix} \right) \\
&= \frac{1 - Y^2}{h^4(t)}.
\end{align*}
The formula for the sectional curvature thus takes the form\footnote{Note that for $h(t) = \exp(-t)$ (for all $t$), we indeed get back the constant curvature $-1$ of the hyperbolic case.}
\begin{align*}
K_M(\Pi) &= - h(t) \cdot \frac{1}{h^2(t)} \cdot h''(t) \cdot Y^2 - h^2(t) \cdot |h'(t)|^2 \cdot \frac{1 - Y^2}{h^4(t)} \\
&= - \frac{h''(t)}{h(t)} \cdot Y^2 - \frac{|h'(t)|^2}{h^2(t)} \cdot (1 - Y^2) .
\end{align*}
Recall that we have $Y^2 \in [0,1]$; furthermore, $h'' > 0$ by properties 1. and 2. of $h$. So in particular, $K_M(\Pi) < 0$ for all $t \in [0,\infty)$. We also have $h'(t) \in [h'(1), h'(3)]$ for $t \in [1,3]$: if we had $h'(t) \notin [h'(1), h'(3)]$ for some $t \in (1,3)$, then $h'$ would have a proper extremum in the intervall $(1,3)$ (i.e. an extremum which is not also assumed on the boundary); but this extremum of $h'$ would mean that $h''$ vanishes at that point, which contradicts $h'' > 0$. Since $h'(1) = -\exp(-1)$ and $h'(3) = -2/3^3$, we thus get
\[
h'(t) \in [-1/e, -2/27] \qquad\text{for } t \in (1,3).
\]
We will now derive lower and upper bounds for the sectional curvature (depending on $t$).

\begin{itemize}
\item {$t\leq 1$}: As the metric hasn't changed here (i.e. $h(t) = \exp(-t)$), we still have $K_M \equiv -1$.

\item {$1<t<3$}: We compute
\begin{align*}
K_M(\Pi) &= - \frac{h''(t)}{h(t)} \cdot Y^2 - \frac{|h'(t)|^2}{h^2(t)} \cdot (1 - Y^2) \\
&\leq - \min\left( \frac{h''(t)}{h(t)}, \frac{|h'(t)|^2}{h^2(t)} \right) \cdot Y^2 - \min\left( \frac{h''(t)}{h(t)}, \frac{|h'(t)|^2}{h^2(t)} \right) \cdot (1 - Y^2) \\
&= - \min\left( \frac{h''(t)}{h(t)}, \frac{|h'(t)|^2}{h^2(t)} \right) \\
&\leq - \min\left( \inf\limits_{1<t<3} \frac{h''(t)}{h(t)}, \inf\limits_{1<t<3} \frac{|h'(t)|^2}{h^2(t)} \right) \\
&\leq - \min\left( \frac{h''(3)}{h(3)}, \frac{2^2/27^2}{h^2(1)} \right) \\
&= - \min\left( \frac{6/3^4}{1/3^2}, \frac{4/729}{1/e^2} \right) \\
&\leq -0.04
\end{align*}
and similarly
\begin{align*}
K_M(\Pi) &\geq - \max\left( \sup\limits_{1<t<3} \frac{h''(t)}{h(t)}, \sup\limits_{1<t<3} \frac{|h'(t)|^2}{h^2(t)} \right) \\
&\geq - \max\left( \frac{h''(1)}{h(1)}, \frac{1/e^2}{h^2(3)} \right) \\
&\geq -11.
\end{align*}

\item {$3\leq t$}: Here $h(t) = 1/t^2$ holds, so we can deduce
\begin{align*}
K_M(\Pi) &= - \frac{h''(t)}{h(t)} \cdot Y^2 - \frac{|h'(t)|^2}{h^2(t)} \cdot (1 - Y^2) \\
&= - \frac{6/t^4}{1/t^2} \cdot Y^2 - \frac{4/t^6}{1/t^4} \cdot (1 - Y^2) \\
&= - \frac{6}{t^2} \cdot Y^2 - \frac{4}{t^2} \cdot (1 - Y^2) \\
&= - \frac{4}{t^2} - \frac{2}{t^2} \cdot Y^2.
\end{align*}
Hence
\[
- \frac{6}{t^2} \leq K_M(\Pi) \leq - \frac{4}{t^2}.
\]
\end{itemize}
\end{proof}

In particular, we have shown that $(M,g')$ satisfies the curvature condition $-11 \leq K_M < 0$; on the other hand, it won't satisfy $K_M \leq a < 0$ for any $a < 0$, as $-6/t^2 \rightarrow 0$ for $t \rightarrow \infty$. Nonetheless, $M$ still is a visibility manifold:

\begin{lem}\label{Lem:Sichtbarkeit_neue_Metrik}
The universal cover of $(M,g')$ satisfies the visibility axiom.
\end{lem}

\begin{proof}
Similar to \cite{JiWu} we will utilize the following fact.

\begin{adjustwidth}{25pt}{0pt}
\textbf{Fact.} (\cite{EO'N} Proposition 5.9) Let $X$ be a Hadamard manifold. If there is a point $x \in X$ such that
\[
\int_1^{\infty} k_V(t) \cdot t \,dt = \infty
\]
for all normalized tangent vectors $V \in T_x X$, then $X$ satisfies the visibility axiom. Here $k_V(t) = \min_{\Pi \ni \dot{c}_V(t)} |K_X(\Pi)|$ denotes the minimum of the (absolute) sectional curvatures in $c_V(t)$ taken over all $2$-dimensional subspaces $\Pi \subseteq T_{c_V(t)} X$ with $\dot{c}_V(t) \in \Pi$, where $c_V$ is the unique geodesic with $c_V(0) = x$ and $\dot{c}_V(0) = V$.
\end{adjustwidth}

Denote the universal cover of $(M,g')$ by $X$. We choose $x \in X$ in the preimage of the boundary of $S$, i.e. $\pi(x) \in \T^{n-1} \times \{0\}$. Let $V \in T_x X$ be an arbitrary normalized tangent vector. Note that by Lemma \ref{Lem:Schnittkruemmung_neue_Metrik}, we have $K_X \leq -0.04/t^2$ for all $t \geq 1$; in particular, $|K_X| \geq 0.04/t^2$ and thus $k_V(t) \geq 0.04/t^2$. So we can compute
\[
\int_1^{\infty} k_V(t) \cdot t \,dt \geq \int_1^{\infty} \frac{0.04}{t^2} \cdot t \,dt = 0.04 \cdot \int_1^{\infty} \frac{1}{t} \,dt = \infty,
\]
and the statement follows by the above fact.
\end{proof}





\end{document}